\documentclass[11pt]{amsart}
\usepackage{amsfonts}
\usepackage{amssymb}
\usepackage{mathrsfs}
\usepackage{srcltx}
\usepackage{subfigure}
\usepackage{amsmath}
\allowdisplaybreaks[4]

\usepackage{amsmath,amsthm,amstext,amsopn,graphicx,cite,xcolor,enumitem}
\newtheorem{theorem}{Theorem}[section]
\newtheorem{lemma}[theorem]{Lemma}
\newtheorem{proposition}[theorem]{Proposition}

\newtheorem{remark}[theorem]{Remark}

\theoremstyle{definition}

\theoremstyle{remark}

\numberwithin{equation}{section}
\newcommand\bes{\begin{eqnarray}}
\newcommand\ees{\end{eqnarray}}
\newcommand\bess{\begin{eqnarray*}}
	\newcommand\eess{\end{eqnarray*}}
\newcommand{\lf}{\left}
\newcommand{\rr}{\right}
\newcommand{\R}{{\mathbb R}}
\newcommand{\dd}{\displaystyle}

\newcommand{\td}{\tilde}
\newcommand{\wtd}{\widetilde}
\newcommand{\wht}{\widehat}
\newcommand\yy{\infty}

\newcommand{\ol}{\overline}
\newcommand{\rd}{{\rm d}}

\newcommand{\prt}{\partial_t}
\newcommand{\prx}{\partial_x}
\newcommand{\prxx}{\partial_{xx}}
\newcommand{\sep}{\sqrt{\epsilon}}
\newcommand{\ul}{\underline}

\begin{document}
 
\date{\today}
\title[Approximation in free boundary problems]{Approximation of random diffusion equation by  nonlocal diffusion equation in free boundary\\ problems of one space dimension$^\$ $}
\author[Y. Du and W. Ni]{Yihong Du$^\dag$ and Wenjie Ni$^\dag$}

\thanks{$^\dag$School of Science and Technology, University of New England, Armidale, NSW 2351, Australia}
	\thanks{$^\$$This research was supported by the Australian Research Council.}
\maketitle

\begin{abstract} We show how the Stefan type free boundary problem with random diffusion  in one space dimension can be approximated by the corresponding free boundary problem with nonlocal diffusion. The approximation problem is a slightly modified version of the nonlocal diffusion problem with free boundaries considered in \cite{cdJFA, CQW}. The proof relies on the introduction of several auxiliary free boundary problems and constructions of delicate upper and lower solutions for these problems. As usual, the approximation is achieved by choosing the kernel function in the nonlocal diffusion term of the form $J_\epsilon(x)=\frac 1\epsilon J(\frac x\epsilon)$ for small $\epsilon>0$, where  $J(x)$ has compact support. We also give an estimate of the error term of the approximation by some positive power of $\epsilon$.

\bigskip

\noindent
{\bf Key words:} Free boundary, random diffusion, nonlocal diffusion, approximation.

\noindent
{\bf MSC2010 subject classifications:} 35K20, 35R35, 35R09.
\end{abstract}

\section {Introduction}

Free boundary problems of the form
	\begin{equation}\label{f2}
\begin{cases}
v_t=dv_{xx} +f(t,x,v), & t>0,\; x\in  (g(t), h(t)),\\
v(t,g(t))= v(t,h(t))=0, &  t >0,\\
g'(t)= -\mu  v_x(t,g(t)), &  t >0,\\
h'(t)= -\mu  v_x(t,h(t)), & t >0,\\
g (0)=-h_0,\; h(0)=h_0,\; v(0,x)=v_0(x), &x\in [-h_0,h_0]
\end{cases}
\end{equation}
have been widely studied in recent years, after the work  \cite{dl2010}, where a logistic type nonlinear term $f=f(v)$ was considered, and the initial function $v_0$
was assumed to satisfy $v_0\in C^2([-h_0, h_0])$, $v_0(\pm h_0)=0$ and $v_0>0$ in $(-h_0, h_0)$. For continuous initial function $v_0$ and general $f=f(t,x,v)$, 
the well-posedness of \eqref{f2}  was proved in \cite{DDL2017}. We refer to \cite{DDL-2, DGP, DLiang, DLou, DMW, DMZ, DWZ, GLZ, KMY, KY, Liang, LLS-1, LLS-2, Wang, WZZ} and the references therein for a sample of the recent works on \eqref{f2}. See also \cite{CF, FS2001, GST2001, RT1989} for some  related earlier works.

If $f\equiv 0$ in \eqref{f2}, then the problem reduces to the well known one-phase Stefan equation \cite{Crank, Robinstein}, which was proposed by Josef Stefan in 1890 to describe the melting of ice in contact with water, and was extensively studied in the past half century; see, for example, \cite{Cafa, CH, Fr, FK, K, KN}.

More recently, the following nonlocal version of \eqref{f2} was proposed and investigated in \cite{cdJFA} (see  \cite{CQW} for the case $f\equiv 0$):
\begin{equation}\label{f1}
\begin{cases}
\dd u_t=d \left[\int_{g(t)}^{h(t)}J(x-y)u(t,y)\rd y-u(t,x)\right]+f(t,x,u), & t>0,\; x\in  (g(t), h(t)),\\
u(t,g(t))= u(t,h(t))=0, & t>0,\\
\dd g'(t)= -\mu \int_{g(t)}^{h(t)}\int_{-\yy}^{g(t)}J(x-y)u(t,x) dyd x, & t>0,\\[3mm]
\dd h'(t)= \mu \int_{g(t)}^{h(t)}\int_{h(t)}^{\yy}J(x-y)u(t,x) dyd x, & t>0,\\
g (0)=-h_0,\; h(0)=h_0,\; u(0,x)=u_0(x), &x\in [-h_0,h_0],
\end{cases}
\end{equation}

In both \eqref{f2} and \eqref{f1}, $\mu$ and $h_0$ are given positive numbers, and for their respective well-posedness, the usual basic assumptions are:

\begin{itemize}
\item
The initial functions $u_0, v_0$ belong to $\mathcal I_0$, where
\[\mathcal I_0:=\{\phi\in C([-h_0,h_0]):  \phi(\pm h_0)=0,\  \phi(x)>0\ {\rm in}\ (-h_0,h_0)\};
\]
\item The function  $f:\R^+\times \R\times \R^+\to \R$ satisfies 
\begin{itemize}
	\item[{\rm ($\mathbf{f_1}$):}]  $f\in C(\R^+\times \R\times \R^+)$,\;  $f(t,x,0)\equiv 0$, $f(t,x,u)$ is locally Lipschitz in $u\in\mathbb R^+$, uniformly in $(t,x)\in \mathbb R^+\times\mathbb R$,
			\item[{\rm ($\mathbf{f_2}$):}] There exists $K>0$ such that $f(t,x,u)\leq 0$ for $u>K$ and $(t,x)\in \R^+\times \R$;
\end{itemize}
\item The kernel function $J:\R\to\R$ in \eqref{f1} is continuous, nonnegative and satisfies
\begin{itemize}
	\item[{\rm ($\mathbf{J}$):}] $J(0)>0$,   $\int_{\R} J(x)\rd x=1$, $J$ is even. 
\end{itemize}
\end{itemize}

In $(\mathbf{f_1})$, the requirement $f(t,x,0)\equiv 0$ can be relaxed to $f(t,x,0)\geq 0$. Assumption $(\mathbf {f_2})$ is a simple sufficient condition  to guarantee that the positive solution stays bounded and hence is defined for all $t>0$. For local existence it is not needed.

For logistic type $f(t,x,u)=f(u)$ (also known as Fisher-KPP type), it was shown in \cite{cdJFA} that the long-time behaviour of \eqref{f1}, similar to \eqref{f2}, is governed by a spreading-vanishing dichotomy. However, when spreading happens, it was proved in \cite{dlz} that the spreading speed of \eqref{f1} could be finite or infinite, depending on the properties of the kernel function $J$ in \eqref{f1}; this is very  different from \eqref{f2}, where the spreading speed is always finite whenever spreading happens (\!\cite{DDL-2, DGP, DLiang, dl2010, DLou, DMZ, KMY, LLS-2}).

For the corresponding fixed boundary problems of \eqref{f2} and \eqref{f1}, it is well-known (\!\cite{AMRT, CER2009, CERW2008, SX2015}) that, over any finite time interval $[0, T]$,  the unique solution $v$ of the local diffusion problem is the limit of the unique solution of the nonlocal problem as $\epsilon\to 0$, when the kernel function $J$ in the nonlocal problem is replaced by 
\[
\tilde J_\epsilon(x)=C J_\epsilon(x):=C\,\frac 1{\epsilon}\,J\left(\frac x {\epsilon}\right)
\]
with a suitable positive constant $C$, provided that $J$ has compact support,  $ f$ and the common initial function are all smooth enough. 

For example, if $J$ satisfies {\bf (J)} with supporting set contained in  $[-1,1]$, and $\tilde J_\epsilon$, $J_\epsilon$  are defined as  above with 
\begin{equation}\label{C*}
C=C_*:=\left[\frac 12 \int_{\R}J(z)z^2dz\right]^{-1}=\left[\int_0^1J(z)z^2dz\right]^{-1},
\end{equation}
 and $F(t,x,u)$ is $C^1$ in $t$, $C^3$ in $(x,u)$, and $u_0\in C^3([a,b])$, then it follows from Theorem A of \cite{SX2015} that the unique solution $u_\epsilon$ of the nonlocal diffusion problem\footnote{Note that this problem is equivalent to
 \[
\begin{cases}
u_t=\dd \int_{\R}  \frac{\tilde J_\epsilon(x-y)}{\epsilon^2}\Big[u(t,y)-u(t,x)\Big]dy+F(t,x,u),& x\in [a,b],\; t>0,\\
u=0,& x\in\R\setminus [a,b],\; t>0,\\
u(0,x)=u_0(x),& x\in [a,b].
\end{cases}
\]
 }
\[
\begin{cases}
u_t=\dd\frac {C_*}{\epsilon^2}\left[\int_a^b  J_\epsilon(x-y)u(t,y)dy-u(t,x)\right]+F(t,x,u),& x\in [a,b],\; t>0,\\
u(0,x)=u_0(x),& x\in [a,b]
\end{cases}
\]
converges to the unique solution $u$ of the corresponding random diffusion problem
\[
\begin{cases}
u_t=u_{xx}+F(t,x,u), & x\in [a,b],\; t>0,\\
u=0, & x\in\{a, b\}, \; t>0,\\
u(0,x)=u_0(x),& x\in [a,b],
\end{cases}
\]
in the following sense: For any $T\in (0, \infty)$,
\[
\lim_{\epsilon\to 0}\|u_\epsilon-u\|_{C([0,T]\times [a,b])}=0.
\]
If $F\equiv 0$ and $u_0\in C^{2+\alpha}([a,b])$, $0<\alpha<1$, then it follows from Theorem 1.1 of \cite{CER2009} that
\[
\|u_\epsilon-u\|_{C([0,T]\times [a,b])}\leq C\epsilon^\alpha
\]
for some $C>0$ and all small $\epsilon>0$.

\smallskip

In this paper, we examine whether  similar results hold between the free boundary problems \eqref{f2} and \eqref{f1}. We show that \eqref{f2} is the limiting problem of a slightly modified version of \eqref{f1}.\footnote{See Remark \ref{rm1.4} below  on the possible necessity of the variation from \eqref{f1}.} The modification occurs in the free boundary equations
\begin{equation}\label{fb}
\begin{cases}
\dd g'(t)= -\mu \int_{g(t)}^{h(t)}\int_{-\yy}^{g(t)}J(x-y)u(t,x) dyd x, \smallskip \\
\dd h'(t)= \mu \int_{g(t)}^{h(t)}\int_{h(t)}^{\yy}J(x-y)u(t,x) dyd x.
\end{cases}
\end{equation}
In \cite{cdJFA}, the equations in \eqref{fb} are obtained from the assumption that the changing population range $[g (t), h(t)]$ of the species with population density $u(t,x)$
expands at each of its end point $(x=g (t)$ and $x=h(t)$) with a rate proportional to the population flux across that end point. 

If we assume instead that these rates are proportional to the population flux across the end points of a slightly reduced region of the population range, say $[g (t)+\delta, h(t)-\delta]$ for some small $\delta>0$, then
\eqref{fb} should be changed accordingly to
\begin{equation}\label{fb1}
\begin{cases}
\dd g'(t)= -\mu \int_{g (t)+\delta}^{h(t)-\delta}\int_{-\yy}^{g (t)+\delta}J(x-y)u(t,x) dyd x, \smallskip \\
\dd h'(t)= \mu \int_{g (t)+\delta}^{h(t)-\delta}\int_{h(t)-\delta}^{\yy}J(x-y)u(t,x) dyd x.
\end{cases}
\end{equation}
So in the context of population spreading as explained in \cite{cdJFA}, the expansion of the population range governed by \eqref{fb1} is also meaningful.

The modified \eqref{f1} then has the form
\begin{equation}\label{f1'}
\begin{cases}
\dd u_t=d \int_{g(t)}^{h(t)}J(x-y)u(t,y)\rd y-du(t,x)+f(t,x,u), & t>0,\; x\in  (g(t), h(t)),\\
u(t,g(t))= u(t,h(t))=0, & t>0,\\
\dd g'(t)= -\mu \int_{g (t)+\delta}^{h(t)-\delta}\int_{-\yy}^{g (t)+\delta}J(x-y)u(t,x) dyd x,& t>0,\\[3mm]
\dd h'(t)= \mu \int_{g (t)+\delta}^{h(t)-\delta}\int_{h(t)-\delta}^{\yy}J(x-y)u(t,x) dyd x, & t>0,\\
g (0)=-h_0,\; h(0)=h_0,\; u(0,x)=u_0(x), & x\in [-h_0,h_0].
\end{cases}
\end{equation}

We are now able to describe the nonlocal approximation problem of \eqref{f2}. Suppose that
\bes\label{J-ep}
\mbox{\rm spt}(J)\subset [-1, 1], \ \  J_\epsilon(x):=\frac{1}{\epsilon}J(\frac{x}{\epsilon}),
\ees
and some extra smoothness conditions on $f$ and $v_0$ (to be specified below) are satisfied. (Here and in what follows, spt$(J)$ stands for the supporting set of $J$.)
Then we will show that the following problem, with $0<\epsilon\ll 1$, is an 
 approximation of \eqref{f2}:
\begin{equation}\label{f3}
\begin{cases}
\dd u_t=d\frac{C_*}{\epsilon^2}\!\! \lf[\int_{g (t)}^{h(t)}\hspace{-.2cm}J_\epsilon(x-y)u(t,y)\rd y-u(t,x)\rr] \!+\!f(t,x,u), & t>0, \; x\in (g (t), h(t)),\\
u(t,g (t))= u (t,h(t))=0, &  t>0,\\
\dd g'(t)= -\mu\frac{C_0}{\epsilon^{3/2}} \int_{g (t)+\sqrt{\epsilon}}^{h(t)-\sqrt\epsilon}\int_{-\yy}^{g (t)+\sqrt{\epsilon}}J_{{\epsilon}}(x-y)u(t,x)\rd y\rd x, & t >0,\\[3mm]
\dd h'(t)=\mu\frac{C_0}{\epsilon^{3/2}} \int_{g (t)+\sqrt\epsilon}^{h(t)-\sqrt{\epsilon} }\int_{h(t)-\sqrt{\epsilon}}^{\yy}J_{{\epsilon}}(x-y)u(t,x)\rd y\rd x, &  t >0,\\
-g (0)=h(0)=h_0,\;u(0,x)=v_0(x), &x\in [-h_0,h_0],
\end{cases}
\end{equation}
where $C_*$ is given by \eqref{C*} and
\begin{equation}\label{1.4}
C_0:= \displaystyle \left[{\int_{-1}^{0}\int_{-1}^{x}J(y)\rd y \rd x}\right]^{-1}= \left[{\int_{0}^{1}\int_{x}^{1}J(y)\rd y \rd x}\right]^{-1}=\left[{\int_{0}^{1}J(y)y\rd y }\right]^{-1}<C_*.
\end{equation}

Let us note that, from \eqref{J-ep} we have $J_\epsilon(x)=0$ for $|x|\geq \epsilon$, and hence, for $0<\epsilon\ll 1$,
\[
\int_{g (t)+\sqrt{\epsilon}}^{h(t)-\sqrt\epsilon}\int_{-\yy}^{g (t)+\sqrt{\epsilon}}J_{{\epsilon}}(x-y)u(t,x)\rd y\rd x
=\int_0^\epsilon\int_{-\epsilon}^0 J_\epsilon(x-y)u(t,g(t)+\sqrt \epsilon+x)dydx,
\]
\[
\int_{g (t)+\sqrt\epsilon}^{h(t)-\sqrt{\epsilon} }\int_{h(t)-\sqrt{\epsilon}}^{\yy}J_{{\epsilon}}(x-y)u(t,x)\rd y\rd x
=\int_0^{\epsilon}\int_{-\epsilon}^0 J_\epsilon(x-y)u(t, h(t)-\sqrt \epsilon-x)dydx.
\]
Therefore in \eqref{f3}, for $0<\epsilon\ll 1$, we may rewrite
\begin{equation}\label{g'-h'}
\begin{cases}
\dd g'(t)= -\mu\frac{C_0}{\epsilon^{3/2}}\int_0^\epsilon\int_{-\epsilon}^0 J_\epsilon(x-y)u(t,g(t)+\sqrt \epsilon+x)dydx,\\
\dd h'(t)=\mu\frac{C_0}{\epsilon^{3/2}}\int_0^{\epsilon}\int_{-\epsilon}^0 J_\epsilon(x-y)u(t, h(t)-\sqrt \epsilon-x)dydx.
 \end{cases}
\end{equation}

The extra smoothness conditions on $f$ and $v_0$ mentioned above are:
There exists some $\alpha\in (0,1)$ such that

\smallskip
$(\mathbf{f_3})$: \ \ \ \ \ \ \ \ \  $f\in C^{\alpha,\alpha,1}(\R^+\times \R\times \R^+)$,  
\begin{align}\label{1.2a}
v_0\in C^{2+\alpha}([-h_0,h_0]),\  v_0(\pm h_0)=0<|v_0'(\pm h_0)|,\;\ v_0(x)>0\ {\rm in}\ (-h_0,h_0).
\end{align}

We are now ready to state the main results of this paper.

\begin{theorem}\label{theorem1.1} Suppose $(\mathbf{f_1})$, $(\mathbf{f_2})$, $(\mathbf{f_3})$, $(\mathbf{J})$ and \eqref{J-ep} hold, and $v_0$ satisfies \eqref{1.2a}.   Then for every small $\epsilon>0$, problem \eqref{f3} has a unique positive solution, denoted by $(u_{\epsilon},g_\epsilon,h_\epsilon)$.  Moreover, if $(v,g,h)$ is the unique positive solution of \eqref{f2} and 	if we define $v(t,x)=0$ for  $x\in \R\setminus (g(t),h(t))$ and  $u_{\epsilon}(t,x)=0$  for  $x\in \R\setminus (g_\epsilon(t),h_\epsilon(t))$,  then, for any $T\in (0, \infty)$,
	\[ \begin{cases}
	\lim_{\epsilon\to 0}\sup_{t\in [0,T]}\|u_{\epsilon}(t,\cdot)-v(t,\cdot)\|_{L^{\yy}(\R)}=0,\\
	 \lim_{\epsilon\to 0}\|g_{\epsilon}-g\|_{L^{\yy}([0,T])}=0,\ \ \lim_{\epsilon\to 0}\|h_{\epsilon}-h\|_{L^{\yy}([0,T])}=0.
	\end{cases}
	\]
\end{theorem}

If we further raise the smoothness requirements on $v_0$ and $f$, namely assuming additionally
\begin{align}\label{1.3a}
v_0\in C^{3+\alpha}([-h_0,h_0]),
\end{align}

($\mathbf{f_4}$):\hspace{3.8cm}$f\in C^{1+\alpha,1+\alpha,1+\alpha}(\R^+\times \R\times \R^+)$, 
\\
then we can obtain an error estimate as follows.

\begin{theorem}\label{theorem1.2} Under the assumptions of Theorem \ref{theorem1.1}, if additionally
 $(\mathbf{f_4})$ and  \eqref{1.3a} are satisfied, then for any $T>0$ and any $\gamma\in (0,\min\{\alpha, \frac 12\})$,   there exists $0<\epsilon_*\ll 1$
		such that for every $\epsilon\in (0,\epsilon_*)$,
		\[ \begin{cases}
		\sup_{t\in[0,T]}\|u_{\epsilon}(t,\cdot)-v(t,\cdot)\|_{L^{\yy}(\R)}\leq \epsilon^{\gamma},\\
		\sup_{t\in[0,T]}|g_{\epsilon}(t)-g(t)|\leq \epsilon^{\gamma},  \ \sup_{t\in [0, T]}|h_{\epsilon}(t)-h(t)|\leq \epsilon^{\gamma}.
		\end{cases}
		\]
\end{theorem}

\begin{remark}\label{rm1.3} {\rm Theorem \ref{theorem1.1} still holds if in \eqref{f3}, the free boundary conditions are changed to, for an arbitrary $\beta\in (0,1)$,
\[\begin{cases}
\dd g'(t)= -\mu\frac{C_0}{\epsilon^{1+\beta}} \int_{g (t)+\epsilon^\beta}^{h(t)-\epsilon^\beta}\int_{-\yy}^{g (t)+\epsilon^\beta}J_{{\epsilon}}(x-y)u(t,x)\rd y\rd x, & t >0,\\[3mm]
\dd h'(t)=\mu\frac{C_0}{\epsilon^{1+\beta}} \int_{g (t)+\epsilon^\beta}^{h(t)-\epsilon^\beta }\int_{h(t)-\epsilon^\beta}^{\yy}J_{{\epsilon}}(x-y)u(t,x)\rd y\rd x, &  t >0,
\end{cases}
\]
or equivalently, in \eqref{g'-h'} the equations are changed to
\begin{equation}\label{g'-h'-beta}
\begin{cases}
\dd g'(t)= -\mu\frac{C_0}{\epsilon^{1+\beta}}\int_0^\epsilon\int_{-\epsilon}^0 J_\epsilon(x-y)u(t,g(t)+\epsilon^\beta+x)dydx,\\
\dd h'(t)=\mu\frac{C_0}{\epsilon^{1+\beta}}\int_0^{\epsilon}\int_{-\epsilon}^0 J_\epsilon(x-y)u(t, h(t)-\epsilon^\beta-x)dydx.
 \end{cases}
\end{equation}
In such a case, Theorem \ref{theorem1.2} still holds if $\gamma\in (0, \min\{\alpha, 1-\beta\})$. Only minor changes are needed in the proofs; for example, for such a case, in \eqref{f4}, $\gamma_1$ should belong to $(\gamma, \min\{\alpha,1-\beta\})$.}
\end{remark}

\begin{remark} \label{rm1.4}{\rm 
We believe that the modification of \eqref{f1} to \eqref{f1'} is necessary in order to obtain an approximation problem of \eqref{f2} such as  \eqref{f3}. Some analysis leading us to this conjecture is given in Section 5.}
\end{remark}

The rest of the paper is organised as follows. In Section 2, we collect some preparatory results for the proof of the main results, and also explain the strategy of the proof (near the end of the section). Sections 3 and 4 consist of  the proofs of Theorems \ref{theorem1.1} and \ref{theorem1.2}, respectively,  based on the construction of delicate upper and lower solutions, following the strategy set in Section 2.
In Section 5, we discuss the conjecture in Remark \ref{rm1.4} through some detailed calculations.

\section{Preparations}
In this section, we prepare some results to be used in the proof of Theorems \ref{theorem1.1} and \ref{theorem1.2} in  Sections 3 and 4. These preparatory results can be proved by simple variations of existing methods and techniques.

\begin{theorem}\label{theorem2.4} 
Suppose $(\mathbf J)$, $(\mathbf{f_1})$ and $(\mathbf{f_2})$ hold,  $u_0\in\mathcal I_0$ and $0\leq \delta\ll h_0$. Then
problem \eqref{f1'}  has  a unique positive solution defined for all $t> 0$. In particular, for $0<\epsilon\ll 1$,  problem \eqref{f3} admits a unique positive solution  $(u_\epsilon, g_\epsilon,h_\epsilon)$ defined for  all $t> 0$.  
\end{theorem} 

\begin{proof}
  In \cite[Theorem 2.1]{cdJFA},  existence and uniqueness for problem \eqref{f1} is proved by using the contraction mapping theorem several times.  If the third and fourth equations of \eqref{f1} are replaced by  \eqref{fb1}, the proof in \cite{cdJFA} can be carried over with only minor and obvious changes. We leave the details to the interested reader.
\end{proof}

\begin{theorem}\label{lemma2.1} Suppose  $(\mathbf{f_1})$ and $(\mathbf{f_2})$ hold, and $u_0\in\mathcal I_0$.
	
		{\rm (i)} Assume that $T\in (0,\yy)$, $0\leq \delta\ll h_0$, and the kernel function $J$ satisfies  $(\mathbf{J})$.   If $(\ol u,\ol g,\ol h)\in C(\ol D)\times  C([0,T]) \times C([0,T]) $ with $D=\{(t,x):t\in (0,T],\; x\in (\ol g(t),\ol h(t))\}$ satisfies 
		\begin{equation}
		\begin{cases}
		\dd \ol u_t\geq d \int_{\ol g(t)}^{\ol h(t)} J(x-y)\ol u(t,y)\rd y-d\ol u(t,x)+f(t,x,\ol u), &t\in (0,T],\; x\in  (\ol g(t), \ol h(t)),\\
		\ol u(t,\ol g(t))\geq 0, \ol u(t,\ol h(t))\geq 0, & t \in (0, T],\\
		\dd \ol g'(t)\leq  -\mu \int_{\ol g(t)+\delta}^{\ol h(t)-\delta}\int_{-\yy}^{\ol g(t)+\delta}J(x-y)\bar u(t,x)\rd y\rd x, & t \in (0, T],\\[3mm]
		\dd \ol h'(t)\geq  \mu \int_{\ol g(t)+\delta}^{\ol h(t)-\delta}\int_{\ol h(t)-\delta}^{\yy}J(x-y)\bar u(t,x)\rd y\rd x, & t \in (0, T],\\
		\ol u(0,x)\geq u_0(x), &x\in [-h_0,h_0]\subset [\ol g(0),\ol h(0)],
		\end{cases}
		\end{equation}
	then 
		\begin{align*}
				& \ol g(t,x)\leq g(t,x),\ \ h(t,x)\leq \ol h(t,x)&&{\rm for}\ t\in [0, T],\\
		&u(t,x)\leq \ol u(t,x)&& {\rm for}\ t\in (0,T],\; x\in  [\ol g(t), \ol h(t)],
		\end{align*}
where $(u,g,h)$ is the unique positive solution of \eqref{f1'}. 
	
	{\rm (ii)} 	Assume  $T\in (0,\yy)$, $\bar g, \bar h, P_1,P_2\in C^1([0,T])$, and  $\ol  v\in  C^{1,2}(\ol D)$ satisfies
		\begin{equation}\label{2.2}
	\begin{cases}
	\ol v_t\geq d\ol v_{xx} +f(t,x,\ol v), \ \ \ \ &t\in (0,T],\; x\in  (\ol g(t), \ol h(t)),\\
	\ol v(t,\ol g(t))= 0,   \ol v(t,\ol h(t))= 0, &t \in (0, T],\\
	\ol g'(t)\leq  -\mu  \ol v_x(t,\ol g(t))+P_1(t), &t \in (0, T],\\
	\ol h'(t)\geq  -\mu  \ol v_x(t,\ol h(t))+P_2(t), &t \in (0, T],\\
	\ol v(0,x)\geq v_0(x), &x\in [-h_0,h_0]\subset [\ol g(0),\ol h(0)].
	\end{cases}
	\end{equation}
If  $(v,g,h)\in C^{1,2}(\ol\Omega)\times C^1([0,T])\times C^1([0,T])$ with  $\Omega=\{(t,x): t\in (0,T],\; x\in (g(t),h(t))\}$ satisfies \eqref{2.2} with all the inequalities replaced by equalities,  then
			\begin{align*}
	&\ol g(t)\leq  g(t), \ \ h(t)\leq \bar h(t)&& {\rm for}\ t\in (0,T],\\
		&v(t,x)\leq \ol v(t,x)&& {\rm for}\ (t,x)\in [0, T]\times [\ol g(t), \ol h(t)].
			\end{align*}
\end{theorem}
\begin{proof}
For  conclusion (i), if $\delta=0$, then it  follows directly from \cite[Lemma 3.1]{cdJFA}. When   $\delta>0$,  one can similarly prove it since the proof of \cite[Lemma 3.1]{cdJFA} is not affected. 

 The comparison principle in part (ii)  can be proved by following the proof of  \cite[Lemma 3.5]{dl2010}, because the extra terms $P_1$ and $P_2$ in the inequalities for ${\ol g}\,'$ and $\ol h\,'$ do not affect the argument there. 
\end{proof}

The  triple $(\ol u,\ol g,\ol h)$ may be called an upper solution. We can define a lower solution by reversing all the  inequality signs and obtain analogous comparison results. 

\medskip
 
Next, for $i=1,2$, $T_0>0$ { and 
	\begin{align}\label{2.3}
	\gamma_1\in (\gamma,\min\{\alpha, \frac 12\})
	\end{align}
	 with $\gamma$ given by Theorem \ref{theorem1.2}}, we consider the following perturbation problems of \eqref{f2}, 
\begin{equation}\label{f4}
\begin{cases}
\prt v_i=d\prxx v_i +f(t,x,v_i)+A_i\epsilon^{\gamma_1}, & t\in (0,T_0],\; x\in ( g_i(t),  h_i(t)),\\
v_i(t,g_i(t))= v_i(t,h_i(t))=0, & t\in (0, T_0],\\
g_i'(t)= -\mu  \prx v_i(t,g_i(t))-B_i\epsilon^{\gamma_1}, & t\in (0, T_0],\\
h_i'(t)= -\mu  \prx v_i(t,h_i(t))+B_i\epsilon^{\gamma_1}, & t\in (0, T_0],\\
 v_i(0,x)= v_0(x), &x\in [-h_0,h_0],
\end{cases}
\end{equation}
where $A_1=B_1=1$,  $A_2=0$ and $B_2=-2$.   We require $A_i\geq 0$ to guarantee the solution is nonnegative and well-defined. Solutions of \eqref{f4} will be used to construct upper and lower solutions of \eqref{f3}. We have the following results.

\begin{theorem}\label{theorem2.2}
{\rm (i)} Suppose $(\mathbf{f_1})$, $(\mathbf{f_2})$ and $(\mathbf{f_3})$ hold, and $v_0$ satisfies \eqref{1.2a}.  Then for any $T_0\in (0,+\infty)$, there exists $\epsilon_0>0$ small depending on $T_0$ such that for any  $\epsilon\in [0,\epsilon_0]$,   problem \eqref{f4} has a unique positive solution $(v_{i\epsilon}, g_{i\epsilon},h_{i\epsilon})$. Moreover, there exists $M_1>0$ depending on  $T_0$, $u_0$ and $\alpha\in (0, 1)$,
such that, for every $\epsilon\in [0,\epsilon_0]$ and $i=1,2$,
\begin{align}\label{2.4a}
||v_{i\epsilon}||_{C^{1+\alpha/2,{2+\alpha}}(\ol \Omega_{i\epsilon})}, \ \  ||g_{i\epsilon}||_{C^{1+\alpha/2}([0,T_0])},||h_{i\epsilon}||_{C^{1+\alpha/2}([0,T_0])}\leq M_1,
\end{align}
where $\Omega_{i\epsilon}:=\{(t,x):t\in (0,T_0],\; x\in ( g_{i\epsilon}(t), h_{i\epsilon}(t))\}.$

{\rm (ii)} If in addition,  $(\mathbf{f_4})$ and  \eqref{1.3a} are satisfied,  then  there exists $M_2>0$ depending on $T_0$ and $u_0$ such that,
 for $ \epsilon\in [0,\epsilon_0]$ and $ i=1,2$,
\begin{align}\label{2.6}
||v_{i\epsilon}||_{C^{\frac{3+\alpha}{2},3+\alpha}(\ol \Omega_{i\epsilon})}, \  ||g_{i\epsilon}||_{C^{2+\alpha/2}([0,T_0])},||h_{i\epsilon}||_{C^{2+\alpha/2}([0,T_0])}\leq M_2.
\end{align}

\end{theorem}

Let us note that if $\epsilon=0$, then \eqref{f4} reduces to \eqref{f2}. The relationship between \eqref{f2} and \eqref{f4} with $\epsilon>0$ is given in the following result.

\begin{theorem}\label{theorem2.3}
Under the assumptions of Theorem \ref{theorem2.2} part {\rm (i)}, the unique solution $(v,g,h)$ of \eqref{f2} satisfies, for $\epsilon\in (0,\epsilon_0]$,
\begin{equation}\label{2.5}
\begin{cases}
[g_{2\epsilon}(t), h_{2\epsilon}(t)]\subset [g(t),h(t)]\subset [g_{1\epsilon}(t), h_{1\epsilon}(t)],\ &t\in [0,T_0],\\
v_{2\epsilon}(t,x)\leq v(t,x)\leq v_{1\epsilon}(t,x),\ &t\in [0,T_0],\; x\in  \R,
\end{cases}
\end{equation}
where we have assumed $v(t,x)=0$  for  $x\in \R\setminus (g(t),h(t))$  and $v_{i\epsilon}(t,x)=0$ for  $x\in \R\setminus (g_{i\epsilon}(t),h_{i\epsilon}(t))$. Moreover,  
\begin{equation}\label{2.6a}
\begin{cases}
&\dd \lim_{\epsilon\to 0}\sup_{t\in [0,T_0]}||v_{i\epsilon}(t,\cdot)-v(t,\cdot)||_{L^{\yy}(\R)}=0,\\
&\dd\lim_{\epsilon\to 0}||g_{i\epsilon}-g||_{L^{\yy}([0,T_0])}=0,\ \ \lim_{\epsilon\to 0}||h_{i\epsilon}-h||_{L^{\yy}([0,T_0])}=0.
\end{cases}
\end{equation}
\end{theorem}

{\bf Proof of Theorem \ref{theorem2.2}:}\ 
This follows from simple variations of existing techniques. So unless necessary, we will be brief and leave the details to the interested reader.
 
 Let the assumptions in part (i) be satisfied.
 For the local existence and uniqueness result to \eqref{f4}, we will follow the proof  of  \cite[Theorem 2.1]{dl2010} with some minor modifications. Since the use of the Sobolev embedding theorem there requires some corrections, we provide the necessary details in this part of the proof. An alternative correction can be found in \cite{Wang2}.

 Firstly, as in \cite{dl2010}, for small $t>0$ we straighten the free boundary of \eqref{f4} by the transformation $(t, y)\to (t,x)$, where
 \[
 x=\Gamma(t,y):=y +\zeta(y)[h(t)-h_0]+\zeta(-y)[g(t)+h_0],\; y\in\R,
\]
 where $(g (t), h(t))$ stands for $(g_i(t), h_i(t))$ with $i=1$ or 2, and $\zeta\in C^\infty(\R)$ satisfies
 \[
 \zeta(y)=1 \mbox{ if } |y-h_0|<\frac {h_0} 4,\; \zeta(y)=0 \mbox{ if } |y-h_0|>\frac{h_0} 2,\; |\zeta'(y)|<\frac{6}{h_0} \mbox{ in } \R.
 \]
 
 Then for small $t>0$, say $t\in [0, S]$ such that $|g (t)-h_0|, |h(t)-h_0|\leq \frac {h_0}{16}$, the interval $[g (t), h(t)]$ in the $x$-axis is changed to $[-h_0, h_0]$ in the $y$-axis, and with 
 \[
 w(t,y):=v_i(t,\Gamma(t,y)),
 \]
 problem  \eqref{f4} for $t\in [0, S]$ is changed to 
\begin{equation}\label{f4'}
\begin{cases}
\prt w=\tilde A(t,y) w_{yy}\!+\!\tilde B(t,y)w_y\!+\!f(t,\Gamma(t,y),w)\!+\!A_i\epsilon^{\gamma_1},\! & t\in (0,S],\; y\in ( -h_0,  h_0),\\
w(t,-h_0)= w(t,h_0)=0, & t\in (0, S],\\
g'(t)= -\mu w_y(t,-h_0) -B_i\epsilon^{\gamma_1}, & t\in (0, S],\\
h'(t)= -\mu w_y(t, h_0) +B_i\epsilon^{\gamma_1}, & t\in (0, S],\\
 w(0,y)= v_0(y), &y\in [-h_0,h_0],
\end{cases}
\end{equation}
 where 
 \[
 \begin{cases}
 \tilde A(t,y)=A(g (t), h(t), y), \\
 \tilde B(t,y)=B(g (t), h(t), y)+g'(t)C_1(g (t), h(t), y)+h'(t)C_2(g (t), h(t), y),
 \end{cases}
 \]
  and $A(\xi,\eta, y)$, $B(\xi,\eta, y)$, $C_1(\xi,\eta, y)$, $C_2(\xi,\eta, y)$
 are $C^\infty$ functions in $ [h_0-\frac{h_0}{16}, h_0+\frac{h_0}{16} ]^2\times \R$, with
 $d/4\leq A(\xi,\eta, y)\leq 16d$ in this range.
 
 Denote 
 \[
 g_1:=-\mu v_0'(-h_0),\; h_1:=-\mu v_0'(h_0), \; S:=\frac{h_0}{16(1+|g_1|+h_1)},
 \]
 and for $T\in (0, S]$ define $\Delta_T:=[0,T]\times [-h_0, h_0]$,
 \[\begin{cases}
 \mathcal D_T:=\{w\in C(\Delta_T): w(0,y)=u_0(y),\; \|w-u_0\|_{C(\Delta_T)}\leq 1\},\\
 \mathcal D_{1T}:=\{g\in C^1([0,T]): g (0)=-h_0,\; g'(0)=g_1,\; \|g'-g_1\|_{C([0,T])}\leq 1\},
 \\
 \mathcal D_{2T}:=\{h\in C^1([0,T]): h(0)=h_0,\; h'(0)=h_1,\; \|h'-h_1\|_{C([0,T])}\leq 1\}.
 \end{cases}
 \]
 It is easily seen that
 $
 \mathcal D_T:=\mathcal D_T\times \mathcal D_{1T}\times \mathcal D_{2T}$ is a complete metric space with the metric
 \[
 d((w, g,h), (\tilde w,\tilde g, \tilde h)):=\|w-\tilde w\|_{C(\Delta_T)}+\|g'-\tilde g'\|_{C([0,T])}+\|h'-\tilde h'\|_{C([0,T])}.
 \]
 
 Given $(w,g,h)\in \mathcal D_T$ with $T\in (0, S]$,  we  extend $(w,g,h)$ to $t> T$ by defining
 \[
 (w(t,y), g (t), h(t))=(w(T,y), g (T), h(T)) \mbox{ for } t>T, \ y\in [-h_0, h_0],
 \]
 and  we extend the associated $\tilde A(t,y)$ and $\tilde B(t,y)$ similarly. For simplicity the extended functions are 
 still denoted by themselves.
 
 Fix $T\in (0, S]$ and for $(w,g,h)\in \mathcal D_T$ we consider the following initial boundary value problem
 \begin{equation}\label{ibvp}
\begin{cases}
\prt \ol w-\tilde A(t,y) \ol w_{yy}\!-\!\tilde B(t,y)\ol w_y\!=\!f(t,\Gamma(t,y),w)\!+\!A_i\epsilon^{\gamma_1},\! & t\in (0,S],\; y\in ( -h_0,  h_0),\\
w(t,-h_0)= w(t,h_0)=0, & t\in (0, S],\\
 w(0,y)= v_0(y), &y\in [-h_0,h_0],
\end{cases}
\end{equation}
 where the above extensions of $(w,g,h)$ and $\tilde A,\, \tilde B$ are assumed for $t>T$.
 
 We note that the modulus of continuity of $\tilde A$, and the $L^\infty$ bound of $\tilde B$ are independent of the choice of $g$ and $h$ above, and $d/4\leq \tilde A\leq 16d$ always holds. Hence,
 by standard $L^p$ theory \eqref{ibvp} has a unique solution $\ol w\in W^{1,2}_p(\Delta_S)$ for any $p>1$, and there exists $C_1>0$ depending only on  $p, \Delta_S$, $\|v_0\|_{C^2([-h_0, h_0])}$, $f$ and $A_i\epsilon^{\gamma_1}$,  such that
 \[
 \|\ol w\|_{W^{1,2}_p(\Delta_S)}\leq C_1.
 \]
 
 By the Sobolev imbedding theorem, for any $\sigma\in (0,1)$, there exists $K_1>0$ and $p>1$ depending on $\sigma$, $h_0$ and $S$ such that
 \[
 \|\ol w\|_{C^{(1+\sigma)/2, 1+\sigma}(\Delta_S)}\leq K_1\|\ol w\|_{W^{1,2}_p(\Delta_S)}\leq C_2:=K_1C_1.
 \]
 It follows that 
\[
 \|\ol w\|_{C^{(1+\sigma)/2, 1+\sigma}(\Delta_T)}\leq C_2.
 \]
 
 Define, for $t\in [0, S]$,
 \begin{equation}\label{h-g}
 \begin{cases}
 \dd\ol h(t)=h_0-\int_0^t\mu\ol w_y(\tau,h_0)d\tau+B_i\epsilon^{\gamma_1} t,
 \\
\dd \ol g (t)=-h_0-\int_0^t\mu\ol w_y(\tau,-h_0)d\tau-B_i\epsilon^{\gamma_1} t. 
\end{cases}
\end{equation}
Then clearly
\[
\|\ol g '\|_{C^{\sigma/2}([0,S])},\; \|\ol h '\|_{C^{\sigma/2}([0,S])}\leq C_3:=\mu C_2+B_i\epsilon^{\gamma_1}.
\]
We now define  $\tilde{\mathcal F}: \mathcal D_T\to C(\Delta_S)\times C^1([0,S])$ and $\mathcal F: \mathcal D_T\to C(\Delta_T)\times C^1([0,T])$ by
\[
\tilde{\mathcal F}(w,g,h)=(\ol w, \ol g,\ol h),\ \ \ \ {\mathcal F}(w,g,h)=(\ol w, \ol g,\ol h)|_{\{t\in [0,T]\}}.
\]
Then the same reasoning as in \cite{dl2010} shows that $\mathcal F$ maps $\mathcal D_T$ into itself if $T\leq S_0\in (0, S)$ for some
$S_0=S_0(C_2,C_3,\sigma)$ small enough.
  
  We show next that by shrinking $T$ further if necessary, $\mathcal F: \mathcal D_T\to\mathcal D_T$ is a contraction mapping.
  Let $(w_j,g_j,h_j)\in \mathcal D_T$ for $j=1,2$ and denote $(\ol w_j,\ol g_j,\ol h_j)=\tilde {\mathcal F}(w_j,g_j, h_j)$. We assume that 
  $(g_j, h_j)$ are extended to $t>T$ as before. We denote the associated $\tilde A(t,y)$ and $\tilde B(t,y)$ by $\tilde A_j(t,y)$ and $\tilde B_j(t,y)$
  and assume that they are also extended to $t>T$ as before. 
  
  Setting $U=\ol w_1-\ol w_2$, we obtain
  \[
  \begin{cases}
  U_t\!-\!\tilde A_2U_{yy}\!-\!\tilde B_2 U_y=&\!\!\!\!(\tilde A_1-\tilde A_2)(\ol w_1)_{yy}+(\tilde B_1-\tilde B_2)(\ol w_1)_{y}\\
  &+f(t,\Gamma_1(t,y), \ol w_1)-f(t,\Gamma_2(t, y),\ol w_2),\ \    t\in (0, S], \; y\in (-h_0, h_0),\\
  U(t,\pm h_0)=0,& t\in [0, S],\\
  U(0,y)=0, & y\in [-h_0, h_0].
    \end{cases}
    \]
  Applying the $L^p$ estimate we obtain, for some $C_4$ depending only on 
  $\Delta_S$ and $p>1$,
  \[
  \|U\|_{W^{1,2}_p(\Delta_S)}\leq C_4(\|w_1-w_2\|_{C(\Delta_T)}+\|g_1-g_2\|_{C^1([0,T])}+\|h_1-h_2\|_{C^1([0,T])}),
  \]
  since the right hand side of the first equation for $U$ above has its $L^p(\Delta_S)$ norm controlled by $\|w_1-w_2\|_{C(\Delta_T)}+\|g_1-g_2\|_{C^1([0,T])}+\|h_1-h_2\|_{C^1([0,T])}$
  due to our extension of the functions, and  the required conditions on $\tilde A_2,\ \tilde B_2$ for the $L^p$ estimate are not affected by the choice of $g_2$ and $h_2$, similar to the situation  for \eqref{ibvp},
  
  We may now use the Sobolev embedding theorem to deduce, as before
  \[
  \|U\|_{C^{(1+\sigma)/2, 1+\sigma}(\Delta_S)}\leq C_5(\|w_1-w_2\|_{C(\Delta_T)}+\|g_1-g_2\|_{C^1([0,T])}+\|h_1-h_2\|_{C^1([0,T])}),
  \]
  with $C_5$ depending only on  $\Delta_S$, $\sigma$ and $C_4$. Therefore, for every $T\in (0, S_0]$,
  \[
  \|\ol w_1-\ol w_2\|_{C^{(1+\sigma)/2, 1+\sigma}(\Delta_T)}\leq C_5(\|w_1-w_2\|_{C(\Delta_T)}+\|g_1-g_2\|_{C^1([0,T])}+\|h_1-h_2\|_{C^1([0,T])}).
  \]
  Using this estimate, we can follow the  argument in \cite{dl2010} to deduce that there exists $S_1\in(0, S_0]$ depending on $C_5$, $\sigma$ and $\mu$ such that $\mathcal F$ is a contraction mapping on $\mathcal D_T$ for any $T\in (0, S_1]$. This guarantees a unique fixed point of $\mathcal F$, which is a positive solution of \eqref{f4} for $t\in (0, S_1]$. 
  
 As in \cite{dl2010}, for any given $T_0>0$, by repeating the above process finitely many times, the positive solution of \eqref{f4} can be extended to $t\in [0, T_0]$, except that
for the case $i=2$, some further explanation is needed, since the extra term $2\epsilon^{\alpha/2}$ in the equations of $h'(t)$ and $g'(t)$
 may cause $h(t)-g(t)$ to decrease in $t$, and the above process requires this quantity to be bounded from below by $h_0$. However, it is easy to show that this lower bound can be guaranteed  over a finite time interval $[0, T_0]$ if $\epsilon>0$ is small enough, say $\epsilon\in (0, \epsilon_0]$.

 We now consider the estimates in \eqref{2.4a}.
 Since the solution over $t\in [0, T_0]$ can be obtained by repeating the local existence proof finitely many times, it is enough to see how the estimates can be obtained over $t\in [0, S_1]$ in the above arguments. With the regularity for $w$, $g$ and $h$ obtained above through the use of $L^p$ theory and Sobolev embedding theorem, the estimates for $g$ and $h$ in \eqref{2.4a} already hold. Moreover, in view of the  assumptions ${\bf (f_3)}$ and \eqref{1.2a}, we see that all the conditions are satisfied to apply the Schauder estimate to \eqref{ibvp} to obtain a $C^{1+\alpha/2, 2+\alpha}$ bound for $w$, which yields a  $C^{1+\alpha/2, 2+\alpha}$ bound for $v$.
 This proves \eqref{2.4a}, and the proof of part (i) is complete.
 
 It remains to prove \eqref{2.6} in part (ii).  We first take $\sigma\in (0, 1)$ in the above arguments so that $\sigma\geq (1+\alpha)/2$. Then
  $w\in C^{(1+\alpha)/2, 1+\alpha}$ and $g,h\in C^{1+\sigma}$  by using the $L^p$ theory and Sobolev embedding theorem in part (i).
 From these facts,  ${\bf (f_4)}$ and \eqref{1.3a}, we see that $\tilde A,\tilde B\in  C^{(1+\alpha)/2, 1+\alpha}$, and the right hand side of the first equation in \eqref{ibvp} belongs to $C^{(1+\alpha)/2, 1+\alpha}$. Hence we can apply the Schauder estimate to \eqref{ibvp} to obtain a $C^{(3+\alpha)/2, 3+\alpha}$ bound for $w$, which yields a $C^{2+\alpha/2}$ bound for $g$ and $h$ through \eqref{h-g}, and then a $C^{(3+\alpha)/2, 3+\alpha}$ bound for $v$. This proves \eqref{2.6}.
 $\hfill \Box$

\medskip

  {\bf Proof of Theorem \ref{theorem2.3}:}\ The validity of \eqref{2.6a} follows from the continuous dependence of the solution of \eqref{f4} with respect to the parameter $\epsilon\in [0,\epsilon_0]$, and \eqref{2.5} follows from the comparison principle. $\hfill \Box$
  
  \begin{remark}\label{rm2.5} {\rm The convergences in \eqref{2.6a} actually hold under stronger norms. For example, combining \eqref{2.6a} with  \eqref{2.4a},
 we immediately see that, for $i=1,2$, 
 \begin{equation*}\label{2.6b}
\dd\lim_{\epsilon\to 0}||g_{i\epsilon}-g||_{C^1([0,T_0])}=0,\ \ \lim_{\epsilon\to 0}||h_{i\epsilon}-h||_{C^1([0,T_0])}=0.
\end{equation*}
 Since $g'(t)<0<h'(t)$ for $t\geq 0$ (here the assumption $v'_0(\pm h_0)\not=0$ is used), the above identities imply that,  for $i=1,2$, \begin{equation}\label{g-h-mono}
 g'_{i\epsilon}(t)<0<h_{i\epsilon}'(t) \mbox{ for all $t\in [0, T_0]$ and all small $\epsilon>0$.}
 \end{equation}
 }
\end{remark}  

\medskip

{\bf Strategy:} We are now in a position to briefly describe the strategy of the proof of the main results. In the next section, we will modify $v_{i\epsilon}$ to obtain $v^*_{i\epsilon}= v_{i\epsilon}+O(\epsilon^{\gamma_1})$ for $i=1,2$ and $0<\epsilon\ll1$ such that
$(v^*_{1\epsilon}, g_{1\epsilon}, h_{1\epsilon})$ is an upper solution of \eqref{f3} and $(v^*_{2\epsilon}, g_{2\epsilon}, h_{2\epsilon})$ is a lower solution of \eqref{f3}. Hence the unique solution $(u_\epsilon, g_\epsilon, h_\epsilon)$ of \eqref{f3} satisfies
\[
v^*_{1\epsilon}\geq v_\epsilon\geq v^*_{2\epsilon},\; g_{1\epsilon}\leq g_{\epsilon}\leq g_{2\epsilon},\;
h_{1\epsilon}\geq h_{\epsilon}\geq h_{2\epsilon}.
 \]
 From these inequalities and \eqref{2.6a}, we immediately obtain the desired conclusions in Theorem \ref{theorem1.1}.
 
 The proof of Theorem \ref{theorem1.2} follows the same strategy, but with $(v_{1\epsilon}, g_{1\epsilon}, h_{1\epsilon})$
 replaced by an upper solution $(V_{1\epsilon}, G_{1\epsilon}, H_{1\epsilon})$ of \eqref{f4} with $i=1$, obtained by modifying the solution $(v_{2\epsilon}, g_{2\epsilon}, h_{2\epsilon})$ of \eqref{f4} with $i=2$, so that $|V_{1\epsilon}-v_{2\epsilon}|+|G_{1\epsilon}-g_{2\epsilon}|+|H_{1\epsilon}-h_{2\epsilon}|$ is bounded by  { $C\epsilon^{\gamma_1}$ for some $C>0$}.

{\bf Notations:} We conclude this section by observing that  ($\mathbf{f_1}$) implies, for any $k>0$ there exists $L_0=L_0(k)>0$ such that
\begin{align}\label{1.2}
|f(t,x,u_1)-f(t,x,u_2)|\leq L_0|u_1-u_2|\ \ \ \ {\rm for}\ u_1,u_2\in [0,k], (t,x)\in \R^+\times \R.
\end{align}
And  ($\mathbf{f_3}$) implies, for any $k>0$ there exists $L=L(k)>0$ such that
\bes\label{1.5}
|f(t_1,x_1,u_1)-f(t_1,x_1,u_2)|\leq L(|t_1-t_2|+|x_1-x_2|+|u_1-u_2|)
\ees
for $u_1,u_2\in [0,k]$, $t_1,t_2\in \R^+$ and $x_1,x_2\in \R$.

\section{Proof of Theorem \ref{theorem1.1}}
Throughout this section we assume that  $(\mathbf{f_1})$, $(\mathbf{f_2})$, $(\mathbf{f_3})$, $(\mathbf{J})$, \eqref{J-ep} and  \eqref{1.2a} hold.  Then, from Section 2, we know that for any $T_0\in (0,+\infty)$, there exists $\epsilon_0>0$ small depending on $T_0$ such that for every  $\epsilon\in (0,\epsilon_0]$ and $i=1,2$,  problem \eqref{f4} has a unique positive solution $(v_{i\epsilon}, g_{i\epsilon},h_{i\epsilon})$.

For $0<\epsilon\ll 1$, let $(u_\epsilon, g_\epsilon, h_\epsilon)$ be the unique solution of \eqref{f3}, which we know from Section 2 is defined for all $t>0$. 

Define 
\begin{equation}\label{m}
m_\epsilon(t,x)=m_\epsilon(t,x; g_{2\epsilon}, h_{2\epsilon}):=\begin{cases}
\dd  1-\left[\frac{x}{g_{2\epsilon}(t)}\right]^2, &t\in [0,T_0],\; x\in [g_{2\epsilon}(t), 0],\\[3mm]
\dd  1-\left[\frac{x}{h_{2\epsilon}(t)}\right]^2, &t\in [0,T_0],\; x\in  [0, h_{2\epsilon}(t)].
\end{cases}
\end{equation}
Clearly $m_\epsilon$ is a $C^1$ function and $0\leq m_\epsilon(t,x)\leq 1$. Moreover, $\partial_x m_\epsilon(t,x)$ is Liptschitz continuous in $x$. 

As we will see below, Theorem \ref{theorem1.1} follows easily from the next result.

\begin{proposition}\label{lemma3.1} There exist positive constants $\tilde M_1$ and $\epsilon_1\in (0,\epsilon_0]$ such that
for any $\epsilon\in (0,\epsilon_1]$, we have
	\begin{equation}\label{ineq3.2}
\begin{cases}
[g_{2\epsilon}(t), h_{2\epsilon}(t)]\subset [g_\epsilon(t),h_\epsilon(t)]\subset [g_{1\epsilon}(t), h_{1\epsilon}(t)],\ &t\in [0,T_0],\\
u_\epsilon(t,x )\geq  v_{2\epsilon}(t,x)-\tilde M_1\epsilon^{\gamma_1},\ &t\in [0,T_0],\; x\in  [g_{2\epsilon}(t),h_{2\epsilon}(t)],\\
u_\epsilon(t,x )\leq v_{1\epsilon}(t,x)+3M_1\epsilon,\ &t\in [0,T_0],\; x\in  [g_\epsilon(t),h_\epsilon(t)],
\end{cases}
\end{equation}
where $M_1$ is given in \eqref{2.4a}.
\end{proposition}

We will use a series of lemmas to prove Proposition \ref{lemma3.1}. Before doing that, let us see how Theorem \ref{theorem1.1} follows easily from Proposition \ref{lemma3.1}.
\smallskip

{\bf Proof of Theorem \ref{theorem1.1}} (assuming Proposition \ref{lemma3.1}): Combining \eqref{ineq3.2} and \eqref{2.6a}, we immediately obtain 
the desired conclusion in Theorem \ref{theorem1.1} with $T=T_0$. $\hfill \Box$
\medskip

We now set to prove Proposition \ref{lemma3.1}.

\begin{lemma}\label{lemma3.1a}    If $L_1\geq L$ with $L=L(M_1+1)$ given by \eqref{1.5} and $M_1$ given by \eqref{2.4a},
and
\[
\;\widehat v_{2\epsilon}(t,x):=v_{2\epsilon}(t,x)-\epsilon^{\gamma_1}Ke^{L_1t}m_\epsilon(t,x) \;\mbox{ with } 
K:=\frac{h_0}{2\mu e^{2L_1T_0}}, 
\]
then there exists $\epsilon_1\in (0,\epsilon_0]$ such that for all $\epsilon\in (0,\epsilon_1]$,
	\begin{align*}
	\wht v_{2\epsilon}(t,x)>0\ \ \mbox{ for }\ \  t\in (0,T_0],\; x\in (g_{2\epsilon}(t),  h_{2\epsilon}(t)),
	\end{align*}
and	 $(\widehat v_2,g_2,h_2)=(\wht v_{2\epsilon}, g_{2\epsilon}, h_{2\epsilon})$ satisfies
	\begin{equation}\label{3.1b}
	\begin{cases}
	\prt \widehat v_2\leq d\prxx \widehat v_2 +f(t,x,\widehat v_2)-\widehat A_2\epsilon^{\gamma_1}, & t\in (0,T_0],\; x\in (g_2(t), h_2(t))\setminus\{0\},\\
	\widehat v_2(t,g_2(t))= \widehat v_2(t,h_2(t))=0, &  t \in (0, T_0],\\
	g_2'(t)\geq  -\mu  \prx \widehat v_2(t,g_2(t))+\epsilon^{\gamma_1}, & t \in (0, T_0],\\
	h_2'(t)\leq -\mu  \prx \widehat v_2(t,h_2(t))-\epsilon^{\gamma_1}, &  t \in (0, T_0],\\
	\widehat v_2(0,x)\leq v_0(x), &x\in [-h_0,h_0],
	\end{cases}
	\end{equation}
	where $\wht A_2:=\frac{2dK}{M_1^2}$.
\end{lemma}
\begin{proof} From the definition of $m_\epsilon$ and $\widehat v_2$, we have $\widehat v_2(t,g_2(t))= \widehat v_2(t,h_2(t))=0$, and $\widehat v_2(0,x)\leq u_0(x)$.  By our assumptions on $u_0$, the definition of $m_\epsilon(t,x)$, and the Hopf boundary lemma applied to $\partial_x v_{2\epsilon}(t, x)$ with $x\in\partial (g_{2\epsilon}(t), h_{2\epsilon}(t))$, we immediately see that
for all small $\epsilon>0$ (depending on $L_1$ and $T_0$), 
\[
\wht v_{2\epsilon}(t,x)>0\ \ \mbox{ for }\ \  t\in [0,T_0],\; x\in (g_{2\epsilon}(t),  h_{2\epsilon}(t)).
\]
 For 
 \[\tilde m=\tilde m_\epsilon:= \epsilon^{\gamma_1}Ke^{L_1t}m_\epsilon(t,x),
 \]
 a simple  computation gives, 
		\begin{equation}\label{3.2c}
	\begin{cases}
	\prt \tilde m=L_1 \tilde m +\epsilon^{\gamma_1} K  e^{L_1 t} 2x^2h_2^{-3}h_2'\geq L_1\tilde m, & x>0,\\
	  \prxx \tilde m=-2\epsilon^{\gamma_1} K e^{L_1 t}h_2^{-2}\leq 0, & x>0,\\
	\prt \tilde m=L_1\tilde m+\epsilon^{\gamma_1} K e^{L_1 t} 2x^2g_2^{-3}g_2'\geq L_1\tilde m, & x<0,\\
	  \prxx \tilde m=-2\epsilon^{\gamma_1} K e^{L_1 t}g_2^{-2}\leq 0,& x<0.
	\end{cases}
	\end{equation}
	Then by \eqref{1.2} and \eqref{f4} we have, provided that $L_1\geq L_0$,
	\begin{align*}
	\prt \widehat v_2=&d\prxx \widehat v_2+f(t,x,\widehat v_2)-\prt \tilde m+d\prxx \tilde m+f(t,x,v_2)-f(t,x,\widehat v_2)\\
   \leq & d\prxx \widehat v_2+f(t,x,\widehat v_2)-L_1\tilde m +d\prxx \tilde m+L_0|\widehat v_2-v_2|\\
\leq & d\prxx \widehat v_2+f(t,x,\widehat v_2)+d\prxx \tilde m\\
\leq &d\prxx \widehat v_2+f(t,x,\widehat v_2)-2d K M_1^{-2}\epsilon^{\gamma_1}\ \mbox{ for }  t\in (0,T_0],\; x\in  (g_2(t), h_2(t))\setminus\{0\},
	\end{align*}
	where we have used $|g_2(t)|$, $h_2(t)\leq M_1$ and \eqref{g-h-mono}.
	
	Next we verify the third and fourth inequalities in \eqref{3.1b}. Applying the fourth equation of \eqref{f4} and $K=\frac{h_0}{2\mu e^{2L_1T_0}}$,  we deduce
	\begin{align*}
	h_2'(t)=& -\mu \prx v_2(t,h_2)-2\epsilon^{\gamma_1}=-\mu  [\prx \widehat v_2(t,h_2)+\prx \tilde m(t,h_2)]-2\epsilon^{\gamma_1}\\
	=&-\mu  \prx \widehat v_2(t,h_2)+2\mu K e^{L_1 t}h_2^{-1}\epsilon^{\gamma_1}-2\epsilon^{\gamma_1}\\
\leq& -\mu  \prx \widehat v_2(t,h_2)+2\mu K e^{2L_1 T_0}h_0^{-1}\epsilon^{\gamma_1}-2\epsilon^{\gamma_1}\\
=& -\mu  \prx \widehat v_2(t,h_2)-\epsilon^{\gamma_1}\ \ \ \ \ \ {\rm for}\ t \in (0, T_0],
	\end{align*}
and  analogously, 
	\begin{align*}
	g_2'(t) \geq   -\mu  \prx w(t,h_2)- \epsilon^{\gamma_1}\ \ \ \ \ \ {\rm for}\ t \in (0, T_0].
	\end{align*}
	The proof is finished.
\end{proof}

Since $v_i=v_{i\epsilon}\in C^{1+\alpha/2, 2+\alpha}(\ol \Omega_i)$ for $i=1,2$ and $\epsilon\in [0,\epsilon_0]$, there exist
 $\wtd v_i=\wtd v_{i\epsilon}\in C^{0,2+{\alpha}}(D)$ for $i=1,2$ and $D=[0,T_0]\times \R$ such that $\wtd v_i=v_i$ in $\Omega_i$. Moreover, in view of \eqref{2.4a}, we may further require, for $i=1,2$ and $\epsilon\in [0,\epsilon_0]$,
	\begin{equation}\label{3.1a}
\dd\sup_{0\leq t\leq T_0}|| \wtd v_i(t,\cdot)||_{C^{2+\alpha}([g_i(t)-\epsilon_0,h_i(t)+\epsilon_0 ])}<2M_1.
\end{equation}
We now define, for $\epsilon\in (0,\epsilon_0]$ and $(t,x)\in D=[0,T_0]\times \R$,
\[
\begin{cases}
v_1^*(t,x)=v_{1\epsilon}^*(t,x):=\wtd v_{1\epsilon}(t,x)+3M_1\epsilon,\\ v_2^*(t,x)=v_{2\epsilon}^*(t,x):=\wtd v_{2\epsilon}(t,x)-\tilde m_\epsilon(t,x).
\end{cases}
\]

Since $v_2(t,g_2(t))=v_2(t,h_2(t))=0$, and by the  Hopf boundary Lemma and the assumptions on $u_0$,  $\prx v_2(t,g_2(t))>0$ and $\prx v_2(t,h_2(t))<0$ for $t\in [0, T_0]$,  for all  sufficiently small $\epsilon>0$, say $\epsilon\in (0,\epsilon_2]\subset (0,\epsilon_1]$,  we have
 $\prx\wtd v_2(t,x)>0$  for $x\in [g_2(t)-\epsilon, g_2(t)]$ and $\prx\wtd v_2(t,x)<0$ for $x\in [h_2(t),h_2(t)+\epsilon]$, which immediately  leads to
\begin{align}\label{3.2a}
v^*_2(t,x)< 0\ \ \ \ \ \ \ \ {\rm for}\ t\in [0,T_0],\; \epsilon\in (0,\epsilon_2],\; x\in (g_2(t)-\epsilon, g_1(t))\cup (h_2(t),h_2(t)+\epsilon).
\end{align}

On the other hand, due to \eqref{3.1a}, by shrinking $\epsilon_2$ if necessary, we have
\begin{align}\label{3.3a}
v_1^*(t,x)>0\ \  {\rm for}\ t\in [0,T_0],\; \epsilon\in (0,\epsilon_2],\; x\in  (g_1(t)-\epsilon, g_1(t))\cup (h_1(t),h_1(t)+\epsilon).
\end{align}

Set
\begin{align*}
&\mathcal L_\epsilon^i[v_i^*](t,x):=\frac{C_*}{\epsilon^2}\lf[\int_{g_i(t)}^{h_i(t)} J_\epsilon(x-y)v_i^*(t,y)\rd y-v_i^*(t,x)\rr],\\
&\wtd{\mathcal L ^i_\epsilon}[v_i^*](t,x):=\frac{C_*}{\epsilon^2}\lf[\int_{g_i(t)-\epsilon}^{h_i(t)+\epsilon} J_\epsilon(x-y)v_i^*(t,y)\rd y-v_i^*(t,x)\rr]. 
\end{align*}
It follows from \eqref{3.2a} and \eqref{3.3a} that, for $t\in [0, T_0]$ and $\epsilon\in (0,\epsilon_2]$,
\begin{equation}\label{L-tdL}
\mathcal L^1_\epsilon [v_1^*](t,x)<\wtd{\mathcal L^1_\epsilon}[v_1^*](t,x),\; \mathcal L_\epsilon^2[v_2^*](t,x)>\wtd{\mathcal L^2_\epsilon}[v_2^*](t,x).
\end{equation}

We show next that $(v_1^*,g_1,h_1)$ (resp.   $(v_2^*,g_2,h_2)$ ) is an upper (resp. a  lower) solution of \eqref{f3}.  

\begin{lemma}\label{lem3.3}
For all small $\epsilon>0$, we have
\begin{equation}\label{3.5c}
\begin{cases}
&\prt  v_1^*\geq   d\mathcal L_\epsilon^1 [v_1^*]+f(t,x,v_1^*)\ \ \ \ \ {\rm for}\ 
t\in [0,T_0],\; x\in (g_1(t),h_1(t)),\\
&\prt  v_2^*\leq  d{\mathcal L_\epsilon^2} [v_2^*]+f(t,x,v_2^*) \ \ \ \ \ {\rm for}\ 
t\in [0,T_0],\; x\in (g_2(t),h_2(t)).
\end{cases}
\end{equation}
\end{lemma}
\begin{proof}
By definition, 
\begin{equation}\label{3.1}
\begin{cases}
v_1^*(t,x)=v_1(t,x)+3M_1 \epsilon,  & t\in [0,T_0],\; x\in  [g_1(t),h_1(t)],\\
v_2^*(t,x)=v_2(t,x)-\tilde m_\epsilon (t,x)=\hat v_2(t,x), &t\in [0,T_0],\; x\in  [g_2(t),h_2(t)],
\end{cases}
\end{equation}
and it follows from \eqref{f4} and \eqref{3.1b} that
\begin{align*}
\prt  v_1^*=&\prt v_1=d\prxx  v_1+f(t,x,v_1)+\epsilon^{\gamma_1}=d\prxx  v_1^*+f(t,x,v_1)+\epsilon^{\gamma_1}\\
=&d\wtd{\mathcal L_\epsilon^1} [v_1^*]+f(t,x,v_1^*)+d\left(\prxx  v_1^*-\wtd{\mathcal L_\epsilon^1} [v_1^*]\right)\\
&+f(t,x,v_1)-f(t,x,v_1^*)+\epsilon^{\gamma_1}\ \ \ \ \ {\rm for}\ 
t\in (0,T_0],\; x\in (g_1(t),h_1(t))
\end{align*}
and
\begin{align*}
\prt  v_2^*=&\prt \wht v_2\leq d\prxx \widehat v_2 +f(t,x,\widehat v_2)-\widehat A_2\epsilon^{\gamma_1}=d\prxx  v_2^* +f(t,x, v_2^*)-\widehat A_2\epsilon^{\gamma_1}\\
=&d\wtd{\mathcal L_\epsilon^2} [v_2^*]+f(t,x,v_2^*)\\
&+d\left(\prxx  v_2^*-\wtd{\mathcal L_\epsilon^2} [v_2^*]\right)-\wht A_2\epsilon^{\gamma_1}\ \ \  \mbox{ for }
t\in (0,T_0],\; x\in (g_2(t),h_2(t))\setminus [-\epsilon,\epsilon].
\end{align*}

Since $\partial_{xx}\tilde m_\epsilon(t,x)$ and hence $\partial_{xx} v^*_2(t,x)$ does not exist at $x=0$,  we estimate $v_2^*(t,x)$ differently
for $x\in [-\epsilon,\epsilon]$.
 As the kernel function in the  operator $\wtd{\mathcal L_\epsilon^2}$ is $J_\epsilon$ whose support is contained in  $[-\epsilon,\epsilon]$ with $0<\epsilon\ll 1$, we have, for $x\in [-\epsilon,\epsilon]$,  
 \[
 \wtd{\mathcal L_\epsilon^2} [v_2^*](t,x)= \wtd{\mathcal L_\epsilon^2} [\wht v_2](t,x)=\wtd{\mathcal L_\epsilon^2} [v_2](t,x)-\wtd{\mathcal L_\epsilon^2} [\tilde m_\epsilon](t,x),
 \]
   and so $v_2^*$ satisfies, for $
(t,x)\in (0,T_0]\times[-\epsilon,\epsilon]$,
\begin{align*}
\prt  v_2^*=&\prt  v_2-\prt \tilde m_\epsilon=d\prxx v_2 +f(t,x, v_2)-\prt \tilde m_\epsilon\\
=&d\wtd{\mathcal L_\epsilon^2} [v_2^*]+f(t,x,v_2^*)  +d\left(\prxx  v_2-\wtd{\mathcal L_\epsilon^2} [v_2]\right)\\
&+\left(d\wtd{\mathcal L_\epsilon^2} [\tilde m_\epsilon]-\prt \tilde m_\epsilon\right) +f(t,x,v_2)-f(t,x,v_2^*).
\end{align*}

We now prove \eqref{3.5c} in several steps.

{\bf Step 1:} We show that, for $t\in (0,T_0],\; x\in (g_1(t),h_1(t))$ and all small $\epsilon>0$, 
\[
\left|\wtd{\mathcal L_\epsilon^1}[v_1^*]-\prxx v_1^*\right|\leq 3M_1\epsilon^{\alpha}.
\]

Recall ${\rm spt} (J_\epsilon)\subset [-\epsilon,\epsilon]$. By Taylor expansion, we obtain, for such $(t,x)$,
\begin{align*}
|\wtd{\mathcal L_\epsilon^1}[v_1^*]-\prxx v_1^*|=&\lf|\frac{C_*}{\epsilon^2} \lf[\int_{g_1(t)-\epsilon}^{h_1(t)+\epsilon} \frac{1}{\epsilon}J(\frac{x-y}{\epsilon})v_1^*(t,y)\rd y-v_1^*(t,x)\rr]-\prxx v_1^*\rr|\\
=&\lf|\frac{C_*}{\epsilon^2} \left[\int_{x-\epsilon}^{x+\epsilon} \frac{1}{\epsilon}J(\frac{x-y}{\epsilon})v_1^*(t,y)dy-v_1^*(t,x)\right]-\prxx v_1^*\rr|\\
=&\lf|\frac{C_*}{\epsilon^2} \int_{-1}^{1} J(z)\big[v_1^*(t,x+\epsilon z)-v_1^*(t,x)\big]\rd z-\prxx v_1^*\rr|\\
= &\lf|\frac{C_*}{\epsilon^2} \int_{-1}^1 J(z)\lf[\epsilon z\prx v_1^*(t,x)+\frac{(\epsilon z)^2\prxx v_1^*(t,x+\delta_1(t,z))}{2}\rr]\rd z-\prxx v_1^*\rr|\\
\leq &\lf|\frac{C_*}{\epsilon^2} \int_{-1}^1 J(z)\lf[\epsilon z \prx v_1^*(t,x)+\frac{1}{2}(\epsilon z)^2\prxx v_1^*(t,x)\rr]\rd z-\prxx v_1^*\rr|\\
&+\frac{C_*}{\epsilon^2} \int_{-1}^1  J(z)\frac{1}{2}(\epsilon z)^2\big|\prxx v_1^*(t,x)-\prxx v_1^*(t,x+\delta_1(t,z))\big|\rd z
\end{align*} 
where  $\delta_1(t,y)$ lies between $0$ and $\epsilon z$.  Due to the symmetry of $J$ and the choice of $C_*$, we have
\[
\frac{C_*}{\epsilon^2} \int_{-1}^1 J(z)\lf[\epsilon z \prx v_1^*(t,x)+\frac{1}{2}(\epsilon z)^2\prxx v_1^*(t,x)\rr]\rd z-\prxx v_1^*
=0.
\]
 Thus
\begin{align*}
|\wtd{\mathcal L_\epsilon^1}[v_1^*]-\prxx v_1^*|\leq &\frac{C_*}{\epsilon^2} \int_{-1}^1  J(z)\frac{1}{2}(\epsilon z)^2|\prxx v_1^*(t,x)-\prxx v_1^*(t,x+\delta_i(t,z))|\rd y\\
\leq &3M_1\epsilon^{\alpha}  {C_*}\int_{-1}^{1} J(z)\frac{1}{2}z^2\rd z= 3M_1\epsilon^{\alpha},
\end{align*} 
where  we have used the estimates in \eqref{3.1a} with $\wtd v_i$ replaced by $v_1^*$. This concludes Step 1.

{\bf Step 2:} We show that, for $t\in (0,T_0],\; x\in (g_2(t),h_2(t))\setminus [-\epsilon,\epsilon]$ and all small $\epsilon>0$, 
\[
|\wtd{\mathcal L_\epsilon^2}[v_2^*]-\prxx v_2^*|\leq 3M_1\epsilon^{\alpha}.
\]

The proof here is identical to that in Step 1, so it is omitted.

\medskip

{\bf Step 3:} We show that, for $t\in (0,T_0],\; x\in [-\epsilon,\epsilon]$ and all small $\epsilon>0$, 
\begin{align*}
\left|\wtd{\mathcal L_\epsilon^2}[v_2]-\prxx v_2\right|\leq 3M_1\epsilon^{\alpha},\; \left|\wtd{\mathcal L_\epsilon^2} [\tilde m_\epsilon]\right |\leq
4\epsilon^{\gamma_1} K e^{L_1 t}h_0^{-2},
\end{align*}
where $K=\frac{h_0}{2\mu e^{2L_1T}}$  is given by Lemma \ref{lemma3.1a}. 

Similarly to Steps 1 and 2, for all small $\epsilon>0$, we have 
\begin{align*}
|\wtd{\mathcal L_\epsilon^2}[v_2]-\prxx v_2|\leq  3M_1\epsilon^{\alpha} \ \ \ {\rm for}\ (t,x)\in (0,T_0]\times[-\epsilon,\epsilon].
\end{align*} 

 It remains to prove the second inequality.
 Using the mean value theorem, by  similar calculations as in Steps 1 and 2, we obtain for $(t,x)\in (0,T_0]\times[-\epsilon,\epsilon]$,
\begin{align*}
\left|\wtd{\mathcal L_\epsilon^2}\,[\tilde m_\epsilon] \right|=&\lf|\frac{C_*}{\epsilon^2} \int_{g_2(t)-\epsilon}^{h_2(t)+\epsilon} \frac{1}{\epsilon}J(\frac{x-y}{\epsilon})[\tilde m_\epsilon (t,y)-\tilde m_\epsilon (t,x)]\rd y\rr|\\
=&\lf|\frac{C_*}{\epsilon^2} \int_{x-\epsilon}^{x+\epsilon} \frac{1}{\epsilon}J(\frac{x-y}{\epsilon})[\tilde m_\epsilon(t,y)-\tilde m_\epsilon (t,x)]\rd y\rr|\\
= &\lf|\frac{C_*}{\epsilon^2} \int_{x-\epsilon}^{x+\epsilon} \frac{1}{\epsilon}J(\frac{x-y}{\epsilon})(y-x)\prx \tilde m_\epsilon(t,\delta(t,y))\rd y\rr|
\\
=&\lf|\frac{C_*}{\epsilon^2} \int_{x-\epsilon}^{x+\epsilon} \frac{1}{\epsilon}J(\frac{x-y}{\epsilon})(y-x)[\prx \tilde m_\epsilon(t,\delta(t,y))-\prx \tilde m_\epsilon(t,x)]\rd y\rr|,
\end{align*} 
where $\delta(t,y)$ lies between $x$ and $y$, and we have used
\[
 \int_{x-\epsilon}^{x+\epsilon} \frac{1}{\epsilon}J(\frac{x-y}{\epsilon})(y-x)dy=0.
\] 
From the definition of $\tilde m_\epsilon$ and \eqref{3.2c}, we see $\prx \tilde m_\epsilon$ is Lipschitz continuous, and 
\begin{align*}
|\prx \tilde m_\epsilon (t,\delta_i(t,y))-\prx \tilde m_\epsilon(t,x)|\leq 2\epsilon^{\gamma_1} K  e^{L_1 t}h_0^{-2}|\delta(t,y)-y|\leq 2\epsilon^{\gamma_1} K e^{L_1 t}h_0^{-2}|x-y|.
\end{align*}
Hence, we can apply the  definition of $C_*$ to deduce 
\begin{align*}
\left|\wtd{\mathcal L_\epsilon^2} [\tilde m_\epsilon]\right |\leq 
&2\epsilon^{\gamma_1} K  e^{L_1 t}h_0^{-2}\frac{C_*}{\epsilon^2} \int_{x-\epsilon}^{x+\epsilon} \frac{1}{\epsilon}J(\frac{x-y}{\epsilon})(y-x)^2 \rd y\\
=&2\epsilon^{\gamma_1} K  e^{L_1 t}h_0^{-2}C_*\int_{-1}^{1} J(z)z^2\rd z=4\epsilon^{\gamma_1} K e^{L_1 t}h_0^{-2}\ \  \ {\rm for}\ (t,x)\in (0,T_0]\times[-\epsilon,\epsilon].
\end{align*} 
This completes Step 3.

{\bf Step 4:} We show that for $t\in (0,T_0],\; x\in [-\epsilon,\epsilon]$ and all small $\epsilon>0$, 
\[
\prt  v_2^*\leq d\wtd{\mathcal L_\epsilon^2} [v_2^*]+f(t,x,v_2^*).
\]

By the identity just before Step 1, it suffices to show, for $t\in (0,T_0],\; x\in [-\epsilon,\epsilon]$ and all small $\epsilon>0$, 
  \[
  d\left(\prxx  v_2-\wtd{\mathcal L_\epsilon^2} [v_2]\right)
+\left(d\wtd{\mathcal L_\epsilon^2} [\tilde m_\epsilon]-\prt \tilde m_\epsilon\right) +f(t,x,v_2)-f(t,x,v_2^*)\leq 0.
\]

By \eqref{1.2}, \eqref{3.2c} and the estimates in Step 3, we have, for $t\in (0,T_0],\; x\in [-\epsilon,\epsilon]$ and all small $\epsilon>0$, 
\begin{align*}
 &d\left(\prxx  v_2-\wtd{\mathcal L_\epsilon^2} [v_2]\right)
+\left(d\wtd{\mathcal L_\epsilon^2} [\tilde m_\epsilon]-\prt \tilde m_\epsilon\right) +f(t,x,v_2)-f(t,x,v_2^*)
\\
\leq &\; 3dM_1\epsilon^{\alpha}+4d K  e^{L_1 t}h_0^{-2}\epsilon^{\gamma_1} -L_1 \tilde m_\epsilon+L_0|v_2^*-v_2|\\
=&-(L_1-L_0)\tilde m_\epsilon+3dM_1\epsilon^{\alpha}+4 dK  e^{L_1 t}h_0^{-2}\epsilon^{\gamma_1}\\
= &\left[-(L_1-L_0)Ke^{L_1 t}m_\epsilon +4 dK  e^{L_1 t}h_0^{-2}\right]\epsilon^{\gamma_1}+3dM_1\epsilon^{\alpha}\\
\leq & \left[ -(L_1-L_0)Ke^{L_1 t}(1-\frac{\epsilon^2}{h_0^2}) +4d K  e^{L_1 t}h_0^{-2}\right]\epsilon^{\gamma_1}+3dM_1\epsilon^{\alpha}\\
\leq &-\left[\frac{1}{2}(L_1-L_0)-4dh_0^{-2}\right]Ke^{L_1 t} \epsilon^{\gamma_1}+3dM_1\epsilon^{\alpha}\\
\leq &-\left[\frac{1}{2}(L_1-L_0)-4dh_0^{-2}\right]K \epsilon^{\gamma_1}+3dM_1\epsilon^{\alpha}<0,
\end{align*}
provided that we first choose $L_1>L_0$ such that $\frac{1}{2}(L_1-L_0)-4dh_0^{-2}>0$ and then choose $\epsilon>0$ sufficiently small.
Step 4 is now completed.

{\bf Step 5:} We show that for $t\in (0,T_0],\; x\in (g_2(t), h_2(t))\setminus [-\epsilon,\epsilon]$ and all small $\epsilon>0$, 
\[
\prt  v_2^*\leq d\wtd{\mathcal L_\epsilon^2} [v_2^*]+f(t,x,v_2^*).
\]

Using our earlier calculation and the estimate in Step 2, we obtain, for  $t\in (0,T_0],\; x\in (g_2(t), h_2(t))\setminus [-\epsilon,\epsilon]$ and all small $\epsilon>0$, 
\begin{align*}
&\prt  v_2^*-d\wtd{\mathcal L_\epsilon^2} [v_2^*]-f(t,x,v_2^*)\\
\leq &\;d\left(\prxx  v_2^*-\wtd{\mathcal L_\epsilon^2} [v_2^*]\right)-\wht A_2\epsilon^{\gamma_1}\\
\leq &\;3dM_1\epsilon^\alpha -\wht A_2\epsilon^{\gamma_1}<0.
\end{align*}
This completes Step 5.

{\bf Step 6:} We show that for $t\in (0,T_0],\; x\in (g_1(t), h_1(t))$ and all small $\epsilon>0$, 
\[
\prt  v_1^*\geq d\wtd{\mathcal L_\epsilon^1} [v_1^*]+f(t,x,v_1^*).
\]

From \eqref{1.2}, we have 
\[
|f(t,x,v_1)-f(t,x,v_1^*)|\leq L_0|v_1-v_1^*|= 3L_0M_1\epsilon.
\]
Combining this with our earlier calculations and the estimate in Step 1, we obtain, for $t\in (0,T_0],\; x\in (g_1(t), h_1(t))$ and all small $\epsilon>0$,
\begin{align*}
&\prt  v_1^*-d\wtd{\mathcal L_\epsilon^1} [v_1^*]-f(t,x,v_1^*)\\
&=d\left(\prxx  v_1^*-\wtd{\mathcal L_\epsilon^1} [v_1^*]\right)+f(t,x,v_1)-f(t,x,v_1^*)+\epsilon^{\gamma_1}\\
&\geq -3dM_1\epsilon^\alpha -3L_0M_1\epsilon +\epsilon^{\gamma_1}>0.
\end{align*}
This completes Step 6.

Clearly \eqref{3.5c} follows directly from \eqref{L-tdL} and the inequalities proved in Steps 4, 5 and 6.
\end{proof}

\begin{lemma}\label{lem3.4}
For all small $\epsilon>0$, we have
	\begin{equation}\label{3.2}
\begin{cases}
\dd h_1'(t)\geq \mu\frac{C_0}{\epsilon^{3/2}} \int_{g_1(t)+\sqrt\epsilon}^{h_1(t)-\sqrt{\epsilon}}\int_{h_1(t)-\sqrt{\epsilon}}^{\yy}J_{\epsilon}(x-y)v_1^*(t,x)\rd y\rd x,\\[3mm]
\dd h_2'(t)\leq  \mu\frac{C_0}{\epsilon^{3/2}} \int_{g_2(t)+\sqrt\epsilon}^{h_2(t)-\sqrt{\epsilon}}\int_{h_2(t)-\sqrt{\epsilon}}^{\yy}J_{\epsilon}(x-y)v_2^*(t,x)\rd y\rd x.
\end{cases}
\end{equation}
\end{lemma}
\begin{proof}
Let us first recall  ${\rm spt} (J)\subset [-1,1]$, $J_{\epsilon}(\xi)=\frac{1}{{{\epsilon}}}J(\frac{\xi}{\epsilon})$ and ${\rm spt} (J_\epsilon)\subset [-\epsilon,\epsilon]$. 
By \eqref{g'-h'} we have 
\begin{align*}
&\mu\frac{C_0}{\epsilon^{3/2}} \int_{g_1(t)+\sqrt\epsilon}^{h_1(t)-\sqrt{\epsilon}}\int_{h_i(t)-\sqrt{\epsilon}}^{\yy}J_{\epsilon}(x-y)v_i^*(t,x)\rd y\rd x\\
=&\;\mu \frac{C_0}{\epsilon^{3/2}}\int_{-\epsilon}^0\int_0^\epsilon J_\epsilon(x-y)v_1^*(t, h_1(t)-\sqrt\epsilon+x)dydx\\
=&\;\mu\frac{C_0}{\sep}  \int_{-1}^{0}\int_{-1}^{w}J(z)v_1^*(t,h_1(t)-{\epsilon} w-\sqrt{\epsilon})\rd z \rd w\\
=&\;\mu\frac{C_0}{{\sep}} \int_{-1}^{0}\int_{-1}^{w}J(z)v_1(t,h_1(t)-{\epsilon} w-\sep)\rd z \rd w+{3\mu_1C_0M_1{\sep}}\int_{-1}^{0}\int_{-1}^{w}J(z)\rd z \rd w\\
=&\;\mu\frac{C_0}{\sep} \int_{-1}^{0}\int_{-1}^{w}J(z)\lf[(-{\epsilon}w-\sep)\prx v_1(t,h_1(t))+\epsilon\frac{(\sep w+1)^2}{2}\prxx v_1(t,h_1+\delta_1(t,w))\rr]\rd z \rd w\\
&+{3\mu C_0M_1{\sep}}\int_{-1}^{0}\int_{-1}^{w}J(z)\rd z \rd w\\
=&\;-\mu \prx v_1(t,h_1(t))+\mu C_0 \sep\int_{-1}^{0}\int_{-1}^{w}J(z) |w|\prx v_1(t,h_1(t))\rd z \rd w\\
&+\mu C_0\sep \int_{-1}^{0}\int_{-1}^{w}J(z)\frac{(\sep w+1)^2}{2}\prxx v_1(t,h_1+\delta_1(t,w))\rd z \rd w\\
&+{3\mu C_0M_1{\sep}}\int_{-1}^{0}\int_{-1}^{w}J(z)\rd z \rd w,
\end{align*}
where $\epsilon w-\sep\leq \delta_1(t,w)\leq 0$ for  $-1\leq w\leq 0$.

Thus, making use of \eqref{2.4a} and $\int_{-1}^{0}J(z)\rd z=\frac{1}{2}$, we get  
\begin{align*}
&\mu\frac{C_0}{\epsilon^{3/2}} \int_{g_1(t)+\sep}^{h_1(t)-\sep}\int_{h_1(t)-\sep}^{\yy}J_\epsilon(x-y)v_1^*(t,x)\rd y\rd x\\
\leq &-\mu \prx v_1(t,h_1(t))+\mu C_0M_1\sep \int_{-1}^{0}\int_{-1}^{w}J(z)[|w|+\frac{1}{2}(\sep w+1)^2]\rd z \rd w\\
&+{3\mu C_0M_1{\sep}}\int_{-1}^{0}\int_{-1}^{w}J(z)\rd z \rd w\\
\leq &-\mu \prx v_1(t,h_1(t))+\mu C_0M_1\sep \int_{-1}^{0}\int_{-1}^{w}J(z)(1+2)\rd z \rd w\\
&+{3\mu C_0M_1{\sep}}\int_{-1}^{0}\int_{-1}^{w}J(z)\rd z \rd w\\
\leq &-\mu \prx v_1(t,h_1(t))+6\mu C_0M_1\sep \int_{-1}^{0}\int_{-1}^{0}J(z)\rd z \rd w\\
= &-\mu \prx v_1(t,h_1(t))+3\mu C_0M_1\sep\\
\leq& -\mu \prx v_1(t,h_1(t))+{\epsilon}^{\gamma_1}= h_1'(t)\ \ \ 
 \ \  \mbox{  for all small }\ \epsilon>0.
\end{align*}

By similar calculations and the fourth inequality of \eqref{3.1b}, we deduce
\begin{align*}
&\mu\frac{C_0}{\epsilon^{3/2}} \int_{g_2(t)+\sep}^{h_2(t)-\sep}\int_{h_2(t)-\sep }^{\yy}J_\epsilon(x-y)v_2^*(t,x)\rd y\rd x\\
\geq  &-\mu \prx \wht v_2(t,h_2(t))-3{\mu C_0}M_1 \sep\int_{-1}^{0}\int_{-1}^{0}J(z){w^2}\rd z \rd w\\
\geq &-\mu \prx \wht v_2(t,h_2(t))-3{\mu C_0}M_1\sep\\
\geq& -\mu \prx \wht v_2(t,h_2(t)) -{\epsilon}^{\gamma_1}\geq  h_2'(t)\ \ \ 
\ \  \mbox{ for all small }\ \epsilon>0.
\end{align*}
Therefore,  \eqref{3.2}  holds.  
\end{proof}

Analogously we can prove
\begin{lemma}\label{lem3.5}
For all small $\epsilon>0$, 
	\begin{equation}\label{3.12}
\begin{cases}
\dd g_1'(t)\leq -\mu \frac{C_0}{\epsilon^{3/2}} \int_{g_1(t)+\sep}^{h_1(t)-\sep}\int_{-\yy}^{g_1(t)+\sep}J_{\epsilon}(x-y)v_1^*(t,x)\rd y\rd x,\\[3mm]
\dd g_2'(t)\geq  -\mu\frac{C_0}{\epsilon^{3/2}} \int_{g_2(t)+\sep}^{h_2(t)-\sep}\int_{-\yy}^{g_2(t)+\sep}J_{\epsilon}(x-y)v_2^*(t,x)\rd y\rd x.
\end{cases}
\end{equation}
\end{lemma}

We are now ready to complete the proof of Proposition \ref{lemma3.1}. 
We note that
\begin{align*}
&v_1^*(t,x)=3M_1\epsilon>0,&& t\in (0,T_0],\; x\in\{g_1(t),h_1(t)\},\\
&v_2^*(t,x)=v_2(x)-\tilde m_\epsilon (t,x)=0,&& t\in (0,T_0], \; x\in\{g_2(t), h_2(t)\},
\end{align*}
and 
\begin{align*}
&v_1^*(0,x)=v_1(0,x)+3M_1\epsilon=v_0(x)+3M_1\epsilon\geq  v_0(x)=u_\epsilon(0,x), &&x\in [-h_0,h_0],\\
&v_2^*(0,x)=v_2(0,x)-m_1(0,x)\leq v_0(x)=u_\epsilon(0,x), &&x\in [-h_0,h_0].
\end{align*}
Hence, in view of the inequalities proved in the previous three lemmas, we can  conclude that $(v_1^*,g_1,h_1)$ (resp.   $(v_2^*,g_2,h_2)$ ) is an upper (resp. a  lower) solution of \eqref{f3}.  Now  the comparison principle in Theorem \ref{lemma2.1} (i), combined with \eqref{3.1}, yields the desired conclusions in \eqref{ineq3.2} with
\[
\tilde M_1:=K e^{L_1T_0}=\frac{h_0}{2\mu e^{L_1T_0}}.
\]

\section{Proof of Theorem \ref{theorem1.2}} In this section, we prove Theorem \ref{theorem1.2}. A crucial step in the proof is to construct an upper solution $(V_{1\epsilon}, G_{1\epsilon}, H_{1\epsilon})$ of \eqref{f4} with $i=1$ by modifying the solution $(v_{2\epsilon}, g_{2\epsilon}, h_{2\epsilon})$ of \eqref{f4} with $i=2$, so that $|V_{1\epsilon}-v_{2\epsilon}|+|G_{1\epsilon}-g_{2\epsilon}|+|H_{1\epsilon}-h_{2\epsilon}|$ is bounded by  {  $C\epsilon^{\gamma_1}$ for some $C>0$}.

To construct $(V_{1\epsilon}, G_{1\epsilon}, H_{1\epsilon})$, for positive constants $\xi_1$ and $\xi_2$, we
 define 
 \[
 \phi(t)=\phi_\epsilon(t):=(1+\xi_2 e^{\xi_1t}\epsilon^{\gamma_1})t,
 \]
  { where $\gamma_1$ is given by \eqref{2.3}}.
  Then for sufficiently large $\xi_1,\;\xi_2, \;\xi_3$ and all small $\epsilon>0$, define
\begin{align*} 
&G_{1\epsilon}(t):=g_{2\epsilon}(\phi(t)), \ \  H_{1\epsilon}(t):=h_{2\epsilon}(\phi(t))& \ \ \  &{\rm for}\ 0<t<\wtd T,\ \\
&V_{1\epsilon}(t,x)=v_{2\epsilon}(\phi(t),x)+\epsilon^{\gamma_1} M_\epsilon(t,x)& \ &{\rm for}\ t\in [0,\wtd T],\; x\in [G_{1\epsilon}(t),H_{1\epsilon}(t)],
\end{align*}
where $(v_{2\epsilon},g_{2\epsilon},h_{2\epsilon})$ is the solution of \eqref{f4} with $i=2$ and $T_0=2T>0$,   $\wtd T>0$ is  uniquely determined by $\phi(\wtd T)=2T$, and
\[
M_\epsilon(t,x):= (\xi_1t+\xi_3)e^{\xi_1 t}m_\epsilon(t,x; G_{1\epsilon}, H_{1\epsilon})= (\xi_1t+\xi_3)e^{\xi_1 t}m_\epsilon(\phi(t),x; g_{2\epsilon}, h_{2\epsilon}),
\]
with  $m_\epsilon(t,x; G_{1\epsilon}, H_{1\epsilon})$ given by \eqref{m} with obvious modifications. Clearly, $M_\epsilon\in C^{1+\alpha/2,1}(\ol \Omega_*)$, where $\Omega_*=\{(t,x): t\in (0,\wtd T],\; x\in (G_{1\epsilon}, H_{1\epsilon})\}$.

\begin{proposition}\label{lemma3.4} Suppose $\mathbf{J}$, $(\mathbf{f_1})$, $(\mathbf{f_2})$ and $(\mathbf{f_4})$ hold, and $u_0$ satisfies \eqref{1.3a}. 
Then there exist $\xi_1,\ \xi_2,\ \xi_3$  large and $\epsilon^*>0$ small such that when  $\epsilon\in (0,\epsilon^*)$,
 the above defined triple  $(V_{1\epsilon}, G_{1\epsilon}, H_{1\epsilon})$  is a weak upper solution\footnote{This is a classical upper solution except on the line $\{(t,0): t>0\}$ where $\partial_xV_{1\epsilon}$ exists and is continuous. Therefore the usual comparison principle holds as for the case of classical upper solutions.}of \eqref{f4} with $i=1$ and $T_0=\tilde T$.
\end{proposition}

Before giving the proof of Proposition \ref{lemma3.4}, let us see how it is used to prove Theorem \ref{theorem1.2}.
\smallskip

{\bf Proof of Theorem \ref{theorem1.2}} (assuming Proposition \ref{lemma3.4}):  It follows from Proposition \ref{lemma3.1} and Proposition \ref{lemma3.4} that 
	\begin{equation}\label{3.8}
	\begin{cases}
	[g_{2\epsilon}(t), h_{2\epsilon}(t)]\subset [g_\epsilon(t),h_\epsilon(t)]\subset [g_{1\epsilon}(t), h_{1\epsilon}(t)],\ &t\in [0,2T],\\
		 u_\epsilon(t,x )\geq  v_{2\epsilon}(t,x)-\tilde M_1 \epsilon^{\gamma_1},\ &t\in [0,2T],\;x\in [g_{2\epsilon}(t),h_{2\epsilon}(t)],\\
	 u_\epsilon(t,x )\leq v_{1\epsilon}(t,x)+3M_1\epsilon,\ &t\in [0,2T],\;x\in [g_\epsilon(t),h_\epsilon(t)],
	\end{cases}
	\end{equation}
 and
 \begin{equation}\label{3.9}
 \begin{cases}
 [g_{1\epsilon}(t),h_{1\epsilon}(t)]\subset \lf[G_{1\epsilon}(t), H_{1\epsilon}(t)\rr], &t\in [0,\wtd T],\\
  v_{1\epsilon}(t,x )\leq V_{1\epsilon}(t,x), &t\in [0,\wtd T],\;x\in [g_{1\epsilon}(t),h_{1\epsilon}(t)].
 \end{cases}
 \end{equation}
 
 {\bf Claim 1}.  For sufficiently small $\epsilon>0$,  we have $T< \wtd T< 2T$ and  
 \begin{align}\label{3.22}
|h_{\epsilon}(t)-h(t)|,\ |g_{\epsilon}(t)-g(t)|\leq  K\epsilon^{\gamma_1}<\epsilon^{\gamma},\ \ \ \  \ t\in [0,T].
 \end{align}
 where {$K=M_2\xi_2 Te^{\xi_1 T}$},   and $(v,g,h)$ is the solution of \eqref{f2}. 
 
 Recall that $\wtd T$ satisfies $\wtd T+\xi_2\wtd Te^{\xi_1 \wtd T}\epsilon^{\gamma_1}=2T$. Then clearly $\wtd T<2T$, and for small $\epsilon>0$,
 \bess
 \wtd T- T=2T-\xi_2\wtd Te^{\xi_1 \wtd T}\epsilon^{\gamma_1} -T\geq T-\xi_2 2Te^{\xi_1 2 T}\epsilon^{\gamma_1}>0.
 \eess
 Taking advantages of  \eqref{2.6}, \eqref{2.5},  \eqref{3.8} and \eqref{3.9},    we deduce for $t\in [0, T]$, 
  \begin{align*}
 |h_{\epsilon}(t)-h(t)|\leq h_{1\epsilon}(t)-h_{2\epsilon}(t)=&\;h_{1\epsilon}(t)-H_{1\epsilon}(t)+h_{2\epsilon}(\phi(t))-h_{2\epsilon}(t)\\
 \leq &\;h_{2\epsilon}(\phi(t))-h_{2\epsilon}(t)\leq ||h_{2\epsilon}'||_{C([0,2T])}|\phi(t)-t|\\
 \leq&\; ||h_{2\epsilon}'||_{C([0,2T])}\xi_2 Te^{\xi_1 T}\epsilon^{\gamma_1}\leq M_2\xi_2 Te^{\xi_1 T}\epsilon^{\gamma_1}\\
 =&\; K\epsilon^{\gamma_1}<\epsilon^{\gamma}
 \ \ \ \ \ \ \ {\rm if}\ 0<\epsilon\ll 1,
 \end{align*}
 and similarly $|g_{\epsilon}(t)-g(t)|\leq  K \epsilon^{\gamma_1}$ for such $\epsilon$. Hence \eqref{3.22} holds.   
 
 Moreover, the above calculations also imply that for $t\in [0,T]$ and $i=1,2$, 
  \begin{align*}
 |h_{i\epsilon}(t)-h(t)|, |g_{i\epsilon}(t)-g(t)|\leq  K\epsilon^{\gamma_1}\  \ \ {\rm if}\ 0<\epsilon\ll 1,
 \end{align*}
 and so, for such $t$, $\epsilon$ and $i=1,2$,  
  \begin{align}\label{3.16}
 [g(t)+K\epsilon^{\gamma_1},h(t)-K\epsilon^{\gamma_1}]\subset [g_\epsilon(t),h_\epsilon(t)]\cap [g_{i\epsilon}(t),h_{i\epsilon}(t)].
 \end{align}

 {\bf Claim 2}. For sufficiently small $\epsilon>0$,  the following estimate holds:
 \begin{align*}
 |u_\epsilon(t,x)-v(t,x)|<\epsilon^{\gamma} \ \mbox{ for }  \ t\in [0,T],\; x\in [g(t)+K\epsilon^{\gamma_1},h(t)-K\epsilon^{\gamma_1}].
 \end{align*}

 By \eqref{2.5}  and \eqref{3.8},  for all small $\epsilon>0$ and $t\in [0,T],\; x\in [g_{2\epsilon}(t),h_{2\epsilon}(t)]$,
  \begin{align*}
 v(t,x)\in [v_{2\epsilon}(t,x),v_{1\epsilon}(t,x)],\ \  u_\epsilon(t,x)\in [v_{2\epsilon}(t,x)-\tilde M_1\epsilon^{\gamma_1},v_{1\epsilon}(t,x)+3M_1\epsilon].
 \end{align*}
 Hence, by  \eqref{3.16}, for all small $\epsilon>0$ and $ \ t\in [0,T],\; x\in [g(t)+K\epsilon^{\gamma_1},h(t)-K\epsilon^{\gamma_1}]$,  we have
  \begin{align}\label{3.23}
 v(t,x),\ u_\epsilon(t,x)\in [v_{2\epsilon}(t,x)-\tilde M_1\epsilon^{\gamma_1},v_{1\epsilon}(t,x)+3M_1\epsilon].
 \end{align}
We may now make use of   \eqref{2.6}, \eqref{3.9} and \eqref{3.16} to conclude that, for all small $\epsilon>0$,
 \begin{align*}
 |u_\epsilon(t,x)-v(t,x)|\leq &\; v_{1\epsilon}(t,x)-v_{2\epsilon}(t,x)+\tilde M_1\epsilon^{\gamma_1}+3M_1\epsilon \\
 \leq &\; V_{1\epsilon}(t,x)-v_{2\epsilon}(t,x)+\tilde M_1\epsilon^{\gamma_1}+3M_1\epsilon \\
 = &\; v_{2\epsilon}(\phi(t),x) -v_{2\epsilon}(t,x)+M_\epsilon(t,x) \epsilon^{\gamma_1}+\tilde M_1\epsilon^{\gamma_1}+3M_1\epsilon \\
 \leq &\; \frac 12 K\epsilon^{\gamma_1}+\|\prt v_{2\epsilon}\|_{L^\yy} |\phi(t)-t| \\
 \leq&\; \frac 12 K\epsilon^{\gamma_1}+M_2\xi_2 Te^{\xi_1 T} \epsilon^{\gamma_1}\\
 \leq&\; \epsilon^{\gamma}\ \ \mbox{ for } \ t\in [0,T],\; x\in[g(t)+K\epsilon^{\gamma_1},h(t)-K\epsilon^{\gamma_1}].
 \end{align*}
  
{\bf Claim 3}. For sufficiently small $\epsilon>0$,  we have
\begin{align*}
|u_\epsilon(t,x)-v(t,x)|<\epsilon^{\gamma} \ \mbox{ for }  \ t\in [0,T],\; x\in \R,
\end{align*}
where $v(t,x)=0$ for  $x\in \R\setminus (g(t),h(t))$ and $u_{\epsilon}(t,x)=0$ for  $x\in \R\setminus (g_\epsilon(t),h_\epsilon(t))$. 

From  \eqref{2.5} and \eqref{3.8}, we see that $[g(t),h(t)]\cup [g_\epsilon(t),h_\epsilon(t)]\subset [g_{1\epsilon}(t),h_{1\epsilon}(t)]$.  Hence in view of Claim 2, we just need to consider the estimate of $|u_\epsilon(t,x)-v(t,x)|$ for $x\in [g_{1\epsilon}(t),h_{1\epsilon}(t)]\setminus[g(t)+K\epsilon^{\gamma_1},h(t)-K\epsilon^{\gamma_1}]$. 

Clearly \eqref{3.23} holds also for $x\in [g_{1\epsilon}(t),h_{1\epsilon}(t)]$  since by our convention $v_{2\epsilon}(t,x)=0$ for  $x\in \R\setminus (g_{2\epsilon}(t),h_{2\epsilon}(t))$. Hence  for $x\in [g_{1\epsilon}(t),h_{1\epsilon}(t)]\setminus[g(t)+K\epsilon^{\gamma_1},h(t)-K\epsilon^{\gamma_1}]$, $t\in [0, T]$ and $0<\epsilon\ll 1$, we have
 \begin{align*}
|u_\epsilon(t,x)-v(t,x)|\leq& v_{1\epsilon}(t,x)-v_{2\epsilon}(t,x)+\epsilon^{\gamma_1}.
\end{align*}
Moreover, taking advantages of \eqref{2.6} and { $h(t)-K\epsilon^{\gamma_1} \leq h_{2\epsilon}\leq h_{1\epsilon}$}, we deduce for  such $t$, $\epsilon$ and $x\in [h(t)-K\epsilon^{\gamma_1},h_{1\epsilon}(t)]$, 
  \begin{align*}
  |u_\epsilon(t,x)-v(t,x)|\leq& \; [v_{1\epsilon}(t,x)-v_{1\epsilon}(t,h_{1\epsilon}(t))]-[v_{2\epsilon}(t,x)-v_{2\epsilon}(t,h_{2\epsilon}(t))]+\epsilon^{\gamma_1}\\
  \leq &\; \|\prx v_{1\epsilon}(t,\cdot)\|_{L^\yy} |x-h_{1\epsilon}(t)|+\|\prx v_{2\epsilon}(t,\cdot)\|_{L^\yy} |x-h_{2\epsilon}(t)|+\epsilon^{\gamma_1}\\
\leq &\;   {[\|\prx v_{1\epsilon}(t,\cdot)\|_{L^\yy} +\|\prx v_{2\epsilon}(t,\cdot)\|_{L^\yy}] [h_{1\epsilon}(t)-h(t)+K\epsilon^{\gamma_1}]+\epsilon^{\gamma_1}}\\
  \leq&\; 4M_2 K\epsilon^{\gamma_1}+  \epsilon^{\gamma_1}\\
  \leq & \; \epsilon^{\gamma}.
  \end{align*}
   Similarly we can show  $|u_\epsilon(t,x)-v(t,x)|\leq \epsilon^{\gamma}$ for $x\in [g_{1\epsilon}(t), g(t)+K\epsilon^{\gamma_1}]$ and $t\in [0, T]$, $0<\epsilon\ll 1$. 
  $\hfill \Box$
\medskip

We now proceed to the proof of Proposition \ref{lemma3.4}. The following technical lemma will be needed.

\begin{lemma}\label{lemma3.3} Suppose the conditions of Proposition \ref{lemma3.4} hold.  
Then $(v_{2\epsilon},g_{2\epsilon},h_{2\epsilon})$, the solution of \eqref{f4} with $i=2$, has the following properties:
	\begin{itemize}
		\item[{\rm (i)}]	There exists $k_1>0$ such that for all small $\epsilon>0$,
		\begin{align}\label{3.9a}
		\prt v_{2\epsilon}(t,x)\geq -{k_1} m_\epsilon(t,x; g_{2\epsilon}, h_{2\epsilon})\ \ \ \ {\rm for}\ t\in [0,T_0],\; 
		x\in [ g_{2\epsilon}(t), h_{2\epsilon}(t)],
		\end{align}
		where  the function $m_\epsilon$ is defined in \eqref{m}. 
		\item[{\rm (ii)}] 	There exists $k_2>0$ such that
		 for all  small  $\epsilon>0$,
	\begin{align}\label{3.10}
		\prx v_{2\epsilon}(t, g_{2\epsilon}(t)) \geq  k_2, \ \ 	 \prx v_{2\epsilon}(t, h_{2\epsilon}(t))\leq - k_2 \ \ \ {\rm for }\ t\in[0,T_0].
	\end{align}
	\end{itemize}
\end{lemma}
\begin{proof} To simplify notations we will write  
\[
\mbox{$(v_2,g_2,h_2)=(v_{2\epsilon},g_{2\epsilon},h_{2\epsilon})$ and $m(t,x)=m_\epsilon(t,x;
g_{2\epsilon}, h_{2\epsilon})$.}
\]

	(i)  By \eqref{2.6}, we check at once that, for $\epsilon\in (0, \epsilon_0]$, 
	\begin{align*}
	&\prx m(t, g_2(t))=\frac{-2x}{ [g_2(t)]^2}\geq \frac{-1}{ g_2(T_0)}\geq  \frac{1}{ M_2}>0,& & t\in [0,T_0],\; x\in [g_2(t),\frac 12 g_2(t)],\\
	&\prx m(t, h_2(t))=\frac{-2x}{[ h_2(t)]^2}\leq  \frac{-1}{ h_2(T_0)}\leq \frac{-1}{M_2}<0,&& t\in [0,T_0],\; x\in [\frac 12 h_2(t),h_2(t)],\\
	&	|\prt v_2(t, x)|\leq  M_2,\ \  |\partial_{tx} v_2(t, x)|\leq M_2,   & & t\in [0,T_0],\; x\in [g_2(t),h_2(t)].
	\end{align*}
	(We note that $|\partial_{tx} v_2(t, x)|\leq M_2$ is the only place in the proof of Theorem \ref{theorem1.2} where \eqref{2.4a} is not enough, and \eqref{2.6} has to be used.)
	
	Due to $v_2(t,g_2(t))=v_2(t,h_2(t))=0$ and $v_2(s,g_2(t)),\; v_2(s,h_2(t))>0$ for $s> t$, we also have 
	\begin{align*}
	 \prt v_2(t, x)\geq 0\ \  {\rm for}\  t\in [0, T_0],\, x=g_2(t)\ {\rm or}\ h_2(t).
	\end{align*}	
Define
\[
P(t,x,k):=km(t,x)+\prt v_2(t,x).
\]
 Then	for any  $k\geq (M_2)^2$,
\begin{align*}
&P(t,g_2(t),k), P(t,h_2(t),k)\geq 0&\ &{\rm for}\  t\in [0, T_0],\\
&\prx P(t,x,k)\geq \frac{k}{M_2}-M_2\geq 0&\ & {\rm for}\  t\in [0,T_0],\; x\in [g_2(t),\frac 12 g_2(t)],\\
&\prx P(t,x,k)\leq  \frac{-k}{ M_2}+M_2\leq  0&\ &{\rm for}\  t\in [0,T_0],\; x\in [\frac 12 h_2(t),h_2(t)],
\end{align*}
which imply
\begin{align*}
&P(t,x,k)\geq 0\ {\rm for}\  t\in [0,T_0],\; x\in \left[g_2(t),\frac 12 g_2(t)\right]\cup \left[\frac 12 h_2(t),h_2(t)\right]. 
\end{align*}

On the other hand, for $(t,x)\in [0, T_0]\times [\frac 12 g_2(t), \frac 12 h_2(t)]$ we have $m(t,x)\geq \frac{3}{4}$ and hence
\begin{align*}
P(t,x,k)\geq \frac{3}{4}k-M_2\geq 0\ {\rm for}\ t\in [0,T_0],\; x\in \left[\frac 12 g_2(t), \frac 12 h_2(t)\right], 
\end{align*}
provided that $k\geq 4M_2/3$.
Therefore, \eqref{3.9a} holds for $k_1\geq \max\{M_2^2,4M_2/3\}$,\; $\epsilon\in (0,\epsilon_0]$. 

	(ii) By Remark \ref{rm2.5}, we have 
		\begin{align*}
	\lim_{\epsilon\to 0}\|g_{2\epsilon}-g\|_{C^1([0,T_0])}=\lim_{\epsilon\to 0}\|h_{2\epsilon}-h\|_{C^1([0,T_0])}=0.
			\end{align*}
		 By the assumption $|u_0'(\pm h_0)|>0$ and 
		 $g'(t)=-\mu v_x(t,g(t))$, $h'(t)=-\mu v_x(t,h(t))$, we obtain $|g'(t)|,\, h'(t)>0$ for all $t\in [0,T_0]$.  Hence,
		 $C:=\min_{t\in [0,T_0]}\{-g'(t), h'(t)\}>0$ and
		  there exists $\tilde\epsilon_0\in (0,\epsilon_0]$ such that 		
		 \begin{align*}
	g_{2\epsilon}'(t)<\frac{g'(t)}{2}< 0,\ \  0<\frac{h'(t)}{2}<{h_{2\epsilon}'(t)}\ \ \ {\rm for\ all}\ t\in[0,T_0],\ \epsilon\in [0,\tilde\epsilon_0]. 
	\end{align*}
	Therefore, if $k_2:=C/(2\mu)$,  then for $ t\in[0,T_0]$ and $ \epsilon\in [0,\tilde\epsilon_0]$,
				\begin{align*}
			&\prx v_{2\epsilon}(t, g_{2\epsilon}(t))=\frac{-g_{2\epsilon}'(t)}{\mu} \geq   \frac{-g'(t)}{2\mu}\geq k_2,\\
	&- \prx v_{2\epsilon}(t, h_{2\epsilon}(t))=\frac{h_{2\epsilon}'(t)}{\mu} \geq \frac{h'(t)}{2\mu}\geq k_2.
				\end{align*}
				This proves \eqref{3.10}. 
\end{proof}

{\bf Proof of Proposition \ref{lemma3.4}:}
By \eqref{f4},  $\td v(t,x)=\tilde v_\epsilon(t,x):=v_{2\epsilon}(\phi(t),x)$ satisfies
	\begin{equation}\label{3.5b}
\begin{cases}
\td v_t=d\td v_{xx}+\frac{\phi'(t)-1}{\phi'(t)}\td v_t+f(\phi(t),x,\td v), &t\in (0,\wtd T],\; x\in (G_{1\epsilon}(t),H_{1\epsilon}(t)),\\
\td v(t,G_{1\epsilon}(t))=\td  v(t,H_{1\epsilon}(t))=0, & t\in (0,\wtd T],\\
G_{1\epsilon}'(t)=  \phi'(t)[-\mu\td v_x(t,G_{1\epsilon}(t))+2\epsilon^{\gamma_1}], & t\in (0,\wtd T],\\
H_{1\epsilon}'(t)= \phi'(t)[-\mu\td v_x(t,H_{1\epsilon}(t))-2\epsilon^{\gamma_1}], & t\in (0,\wtd T],\\
\td v(0,x)=v_0(x), &x\in [g_0,h_0].
\end{cases}
\end{equation}

From the first equation of \eqref{3.5b} and \eqref{1.5}, we obtain
\begin{align*}
(V_{1\epsilon})_t=&\; \prt \td v+\epsilon^{\gamma_1}\prt M_\epsilon=d\td v_{xx}+\frac{\phi'(t)-1}{\phi'(t)}\td v_t+f(\phi(t),x,\td v)+\epsilon^{\gamma_1}\prt M_\epsilon\\
=&\; d (V_{1\epsilon})_{xx}+f(t,x,V_{1\epsilon})+\frac{\phi'(t)-1}{\phi'(t)}\td v_t\\
&+\epsilon^{\gamma_1}\prt M_\epsilon-d\epsilon^{\gamma_1}\prxx M_\epsilon +f(\phi(t),x, \td v)-f(t,x,V_{1\epsilon})\\
\geq&\; d (V_{1\epsilon})_{xx}+f(t,x,V_{1\epsilon})+\frac{\phi'(t)-1}{\phi'(t)}\td v_t\\
&+\epsilon^{\gamma_1}\prt M_\epsilon-d\epsilon^{\gamma_1}\prxx M_\epsilon-L(\epsilon^{\gamma_1} M_\epsilon+|\phi(t)-t|)\\
=&:d (V_{1\epsilon})_{xx}+f(t,x,V_{1\epsilon})+E(t,x)\   \ \ {\rm for} \ t\in (0,\wtd T],\; x\in (G_{1\epsilon}(t),0)\cup(0,H_{1\epsilon}(t)).
\end{align*}

 A simple computation gives
	\begin{equation}\label{3.16b}
\begin{cases}
\prt M_\epsilon \geq  \xi_1(\xi_1t+\xi_3) e^{\xi_1 t}\left[1-\frac{x^2}{H_{1\epsilon}^2(t)}\right]= \xi_1M_\epsilon &{\rm for}\ x>0,\\
\prxx M_\epsilon= -2 (\xi_1t+\xi_3)e^{\xi_1 t}\frac 1{H^2_{1\epsilon}(t)}<0& {\rm for}\ x>0,\\
\prt M_\epsilon \geq   \xi_1(\xi_1 t+\xi_3) e^{\xi_1 t} \left[1-\frac{x^2}{G_{1\epsilon}^2(t)}\right] =\xi_1M_\epsilon&{\rm for}\ x<0,\\
\prxx M_\epsilon= -2 (\xi_1 t+\xi_3)e^{\xi_1 t}\frac 1{G^2_{1\epsilon}(t)}<0,
 & {\rm for}\ x<0.
\end{cases}
\end{equation}

{\bf Claim 1}. We can choose $\xi_1, \ \xi_2$ and $\xi_3$ such that 
\[\mbox{$E(t,x)\geq \epsilon^{\gamma_1}$ for $ t\in (0,\wtd T],\; x\in (G_{1\epsilon}(t),0)\cup(0,H_{1\epsilon}(t))$ and $0<\epsilon\ll 1$.}
\]

In the following, we just verify $E(t,x)\geq \epsilon^{\gamma_1}$ for $t\in (0,\wtd T],\; x\in (0,H_{1\epsilon}(t))$ since the proof for $t\in (0,\wtd T],\; x\in (G_{1\epsilon}(t),0)$ is similar. 

Since $m_\epsilon(t,x;G_{1\epsilon}(t),H_{1\epsilon}(t))=m_\epsilon(\phi(t),x;g_{2\epsilon},h_{2\epsilon})$,
from \eqref{3.9a} we deduce
\begin{align*}
\prt v_{2\epsilon}(\phi,x)
\geq-k_1m_\epsilon(t,x;G_{1\epsilon},H_{1\epsilon})
\end{align*}
where $k_1$ is given by Lemma \ref{lemma3.3}, and so,   for $0<\epsilon\ll 1$ and $ t\in (0,\wtd T],\; x\in [G_{1\epsilon}(t),H_{1\epsilon}(t)]$,
we have
\begin{align}\label{3.6}
\td v_t(t,x)=\phi'\prt v_{2\epsilon}(\phi,x)
\geq -2k_1m_\epsilon(t,x;G_{1\epsilon},H_{1\epsilon}),
\end{align}  
where we have used 
\begin{align}\label{3.17b}
0<\xi_2\epsilon^{\gamma_1}<\phi'(t)-1=\xi_2 (1+\xi_1t)e^{\xi_1t}\epsilon^{\gamma_1}<1\ \ \ \  {\rm for\ small}\ \epsilon>0.
\end{align}

 Making use of  \eqref{3.16b}, \eqref{3.6} and \eqref{3.17b}, we obtain, for $(t,x)\in (0,\wtd T]\times (0,H_{1\epsilon}(t))$ and $0<\epsilon\ll 1$,
\begin{align*}
E=&\; \frac{\phi'(t)-1}{\phi'(t)}\td v_t+\epsilon^{\gamma_1}\prt M_\epsilon-d\epsilon^{\gamma_1}\prxx M_\epsilon-L(\epsilon^{\gamma_1} M_\epsilon+|\phi(t)-t|)\\
\geq &\; -2k_1[\phi'(t)-1]m_\epsilon+\xi_1M_\epsilon\epsilon^{\gamma_1}+2d(\xi_1t+\xi_3)\frac{e^{\xi_1t}}{H_{1\epsilon}^2(t)}\epsilon^{\gamma_1}-LM_\epsilon\epsilon^{\gamma_1}-L|\phi(t)-1|\\
\geq&\; -2k_1 \xi_2M_\epsilon \epsilon^{\gamma_1}+  \xi_1M_\epsilon\epsilon^{\gamma_1}+2d(\xi_1t+\xi_3)\frac{e^{\xi_1t}}{H_{1\epsilon}^2(t)}\epsilon^{\gamma_1} -LM_\epsilon\epsilon^{\gamma_1}
-L\xi_2t e^{\xi_1 t} \epsilon^{\gamma_1}\\
=&\;  (\xi_1-2k_1\xi_2-L)M_\epsilon \epsilon^{\gamma_1}+\left[2d(\xi_1t+\xi_3)\frac{e^{\xi_1t}}{H_{1\epsilon}^2(t)}-L\xi_2t e^{\xi_1 t}\right]\epsilon^{\gamma_1}.
\end{align*} 

 Since $H_{1\epsilon}(\wtd T)=h_{2\epsilon}(T_0)\leq M_2$, we have
 \[
 2d(\xi_1t+\xi_3)\frac{e^{\xi_1 t}}{H_{1\epsilon}^2(t)}-L\xi_2t e^{\xi_1 t}\geq \frac{2d}{M_2^2}\xi_3+
 \left(\frac{2d}{M_2^2}\xi_1-L\xi_2\right)te^{\xi_1 t}>1
 \]
 provided that
 \[
 \xi_1> \frac{LM_2^2}{2d}\xi_2,\;\ \  \xi_3>\frac{M_2^2}{2d}.
 \]
 Therefore, for $(t,x)\in (0,\wtd T]\times (0,H_{1\epsilon}(t))$ and $0<\epsilon\ll 1$,
\begin{align*}
E(t,x)>\epsilon^{\gamma_1}
\end{align*}
provided that
\begin{equation}\label{xi-123}
\xi_1>\max\left\{2k_1\xi_2+L, \frac{LM_2^2}{2d}\xi_2\right\},\ \ \ \  \xi_3>\frac{M_2^2}{2d}.
\end{equation}
Claim 1 is thus proved, and hence, for such $\epsilon, \;\xi_1$, $\xi_2$ and $\xi_3$, we have
\begin{align*}
(V_{1\epsilon})_t\geq d (V_{1\epsilon})_{xx}+f(t,x,V_{1\epsilon})+\epsilon^{\gamma_1}\   \ \ {\rm for} \ t\in (0,\wtd T],\; x\in (G_{1\epsilon}(t),0)\cup(0,H_{1\epsilon}(t)).
\end{align*}

Next, we deal with the estimates of $G_{1\epsilon}'$ and  $H_{1\epsilon}'$. From the forth equation of \eqref{3.5b}, \eqref{3.17b} and $V_{1\epsilon}=\td v+\epsilon^{\gamma_1} M_\epsilon$, we obtain
\begin{align*}
H_{1\epsilon}'(t)=&\phi'(t)[-\mu\td v_x(t,H_{1\epsilon}(t))-2\epsilon^{\gamma_1}]\geq  -\mu\phi'(t) \td v_x(t,H_{1\epsilon}(t))-4\epsilon^{\gamma_1}\\
=&-\mu \td v_x(t,H_{1\epsilon}(t)) -\mu[\phi'(t)-1] \td v_x(t,H_{1\epsilon}(t))-4\epsilon^{\gamma_1}\\
=&-\mu [(V_{1\epsilon})_x(t,H_{1\epsilon}(t))-\epsilon^{\gamma_1}\prx M_\epsilon(t,H_{1\epsilon}(t))]-\mu[\phi'(t)-1] \td v_x(t,H_{1\epsilon}(t))-4\epsilon^{\gamma_1}\\
=& -\mu(V_{1\epsilon})_x(t,H_{1\epsilon}(t))-\mu[\phi'(t)-1]\td v_x(t,H_{1\epsilon}(t))+\mu \prx M_\epsilon(t,H_{1\epsilon}(t))\epsilon^{\gamma_1}-4\epsilon^{\gamma_1}\\
=:&-\mu(V_{1\epsilon})_x(t,H_{1\epsilon}(t))+E_1(t).
\end{align*}

{\bf Claim 2}.  We can choose $\xi_1, \ \xi_2$ and $\xi_3$ satisfying \eqref{xi-123} such that
\[
\mbox{$E_1(t)\geq \epsilon^{\gamma_1}$ for $ t\in (0,\wtd T]$ and $0<\epsilon\ll 1$.}
\]

 With $k_2$ determined by Lemma \ref{lemma3.3}, by \eqref{3.10},  we have
\bes\label{3.7}
-\td v_x(t,H_{1\epsilon}(t)) \geq k_2\ \ \ \ \ \ {\rm for }\ 0\leq t\leq \wtd T.
\ees
Then applying \eqref{3.7} and $H_{1\epsilon}(t)\geq h_0$,  we deduce, for $0\leq t\leq \wtd T$ and $0<\epsilon\ll 1$,
\begin{align*}
E_1=&\;-\mu[\phi'(t)-1]\td v_x(t,H_{1\epsilon}(t))+\mu \prx M_\epsilon(t,H_{1\epsilon}(t))\epsilon^{\gamma_1}-4\epsilon^{\gamma_1}\\
\geq &\; \mu[\phi'(t)-1]k_2 -2\mu ( \xi_1 t+\xi_3)e^{\xi_1 t}H_{1\epsilon}^{-1}\epsilon^{\gamma_1}-4\epsilon^{\gamma_1}\\
= &\; \mu \xi_2(1+\xi_1 t) e^{\xi_1 t} k_2\epsilon^{\gamma_1}-2\mu (\xi_1 t+\xi_3)e^{\xi_1 t}H_{1\epsilon}^{-1}\epsilon^{\gamma_1}-4\epsilon^{\gamma_1}\\
=&\;(\xi_2 k_2-2\xi_3H_{1\epsilon}^{-1})\mu e^{\xi_1 t}\epsilon^{\gamma_1}+(\xi_2-2H_{1\epsilon}^{-1})\mu\xi_1 t e^{\xi_1t}\epsilon^{\gamma_1}-4\epsilon^{\gamma_1}\\
\geq &\; (\xi_2k_2-2\xi_3h_0^{-1})\mu  \epsilon^{\gamma_1}+(\xi_2-2h_0^{-1})\mu\xi_1 t e^{\xi_1t}\epsilon^{\gamma_1}-4\epsilon^{\gamma_1}\\
\geq &\; 5 \epsilon^{\gamma_1}-4\epsilon^{\gamma_1} =\epsilon^{\gamma_1},
\end{align*}
provided that
\begin{equation}\label{xi-23}
{\xi_2\geq \max\left\{\frac 2 {h_0},  \frac{5}{\mu k_2}+\frac{2}{h_0k_2}\xi_3\right\}.}
\end{equation}

This proves Claim 2 and hence, for  $\xi_1,\xi_2,\xi_3$ satisfying \eqref{xi-123} and \eqref{xi-23}, and $0<\epsilon\ll 1$,
\begin{align*}
H_{1\epsilon}'(t)\geq  -\mu(V_{1\epsilon})_x(t,H_{1\epsilon}(t))+\epsilon^{\gamma_1},\ \ \ \ \ \ t\in [0,\wtd T].
\end{align*}
Analogously, for such  $\xi_1,\xi_2,\xi_3$ and $0<\epsilon\ll 1$,
\begin{align*}
G_{1\epsilon}'(t)\leq -\mu(V_{1\epsilon})_x(t,G_{1\epsilon}(t))-\epsilon^{\gamma_1},\ \ \ \ \ \ \ \ t\in [0,\wtd T].
\end{align*}
Furthermore, from the definition of $M_\epsilon$,
\begin{align*}
&V_{1\epsilon}(t,G_{1\epsilon}(t))=V_{1\epsilon}(t,H_{1\epsilon}(t))=0 && {\rm for}\ t\in  [0,\wtd T],\\
&V_{1\epsilon}(0,x)=\td v(0,x)+\epsilon^{\gamma_1} M_\epsilon(0,x)\geq u_0(x)&&{\rm for}\ x\in [-h_0,h_0].
\end{align*}
Since $M_\epsilon(t,x)$ and hence $V_{1\epsilon}(t,x)$ is $C^1$ at $x=0$, we may now conclude that $(V_{1\epsilon},G_{1\epsilon},H_{1\epsilon})$ is a weak upper solution of \eqref{f4} with $i=1$.
$\hfill \Box$

\section{About Remark \ref{rm1.4}} 
Here
we provide some analysis which leads us to believe the modification of \eqref{f1} is needed for the approximation of \eqref{f2}. 

Without modifying \eqref{f1}, the natural candidate for the approximation problem of \eqref{f2} is the following one:
\begin{equation}\label{f3'}
\begin{cases}
\dd u_t=d\frac{C_*}{\epsilon^2}\!\! \lf[\int_{g (t)}^{h(t)}\hspace{-.2cm}J_\epsilon(x-y)u(t,y)\rd y-u(t,x)\rr] \!+\!f(t,x,u), & t>0, \; x\in (g (t), h(t)),\\
u(t,g (t))= u (t,h(t))=0, &  t>0,\\
\dd g'(t)= -\mu\frac{C_1}{\epsilon^{2}} \int_{g (t)}^{h(t)}\int_{-\yy}^{g (t)}J_{{\epsilon}}(x-y)u(t,x)\rd y\rd x, & t >0,\\[3mm]
\dd h'(t)=\mu\frac{C_1}{\epsilon^{2}} \int_{g (t)}^{h(t) }\int_{h(t)}^{\yy}J_{{\epsilon}}(x-y)u(t,x)\rd y\rd x, &  t >0,\\
-g (0)=h(0)=h_0,\;u(0,x)=v_0(x), &x\in [-h_0,h_0]
\end{cases}
\end{equation}
for some $C_1>0$, where $C_*=\displaystyle \left[{\int_{0}^{1}J(y)y^2\rd y}\right]^{-1}$ as in \eqref{C*}.

\begin{proposition}\label{p5.1}
Suppose the conditions of Theorem \ref{theorem1.1} hold, $(v,g,h)$ is the solution of \eqref{f2}, and  $(u_\epsilon, g_\epsilon, h_\epsilon)$ is the solution of \eqref{f3'}.  If $(u_\epsilon, g_\epsilon, h_\epsilon)\to (v,g,h)$ as $\epsilon\to 0$ uniformly for $x\in\R$ and $t\in [0, T]$ for every $T>0$, then
\begin{align*}
C_1=C_*.
\end{align*}
\end{proposition}
\begin{proof}
A simple calculation gives
\begin{align*}
\frac{\partial}{\partial t} \lf[\int_{g_\epsilon(t)}^{h_\epsilon(t)} u_\epsilon(t,x) \rd x\rr]=&h_\epsilon'(t)u_\epsilon(t,h_\epsilon(t))-g_\epsilon'(t)u_\epsilon(t,g_\epsilon(t))+\int_{g_\epsilon(t)}^{h_\epsilon(t)} \partial_t u_\epsilon(t,x) \rd x\\
=&\int_{g_\epsilon(t)}^{h_\epsilon(t)} \partial_t u_\epsilon(t,x) \rd x
\end{align*}
and 
\begin{align*}
&\int_{g_\epsilon(t)}^{h_\epsilon(t)}\lf[\int_{g_\epsilon(t)}^{h_\epsilon(t)} J_\epsilon(x-y)u_\epsilon(t,y)\rd y-u_\epsilon(t,x) \rr]\rd x\\
=&\int_{g_\epsilon(t)}^{h_\epsilon(t)}\int_{g_\epsilon(t)}^{h_\epsilon(t)} J_\epsilon(x-y)[u_\epsilon(t,y)-u_\epsilon(t,x)]\rd y\rd x\\
& -\int_{g_\epsilon(t)}^{h_\epsilon(t)}\int_{\R\setminus [g_\epsilon(t), h_\epsilon(t)]}J_\epsilon(x-y) u_\epsilon(t,x)\rd y\rd x\\
=&-\int_{g_\epsilon(t)}^{h_\epsilon(t)}\int_{h_\epsilon(t)}^{\yy}J_\epsilon(x-y) u_\epsilon(t,x)\rd y\rd x-\int_{g_\epsilon(t)}^{h_\epsilon(t)}\int_{-\yy}^{g_\epsilon(t)}J_\epsilon(x-y) u_\epsilon(t,x)\rd y\rd x\\
=& -\frac{\epsilon^2}{\mu C_1}[h_\epsilon'(t)-g_\epsilon'(t)]. 
\end{align*}
Integrating the first equation in \eqref{f3'} over $\{(x,s): x\in (g_\epsilon(s),h_\epsilon(s)), s\in (0,t)\}$, we thus obtain
\begin{align*}
\int_{g_\epsilon(t)}^{h_\epsilon(t)} u_\epsilon(t,x){\rm d}x
-\int_{g_0}^{h_0} v_0(x) {\rm d}x&=\int_0^t \int_{g_\epsilon(s)}^{h_\epsilon(s)} \partial _t u_\epsilon(s,x)dxds\\
& =-\frac{dC_*}{\mu C_1}[h_\epsilon(t)-g_\epsilon(t)-2h_0]+\int_{0}^{t}\int_{g_\epsilon(s)}^{h_\epsilon(s)}  f(s,x,u_\epsilon) {\rm d} x{\rm d} s.
\end{align*}
Letting $\epsilon\to 0$ we deduce
\begin{align}\label{limit-u}
\int_{g(t)}^{h(t)} v(t,x){\rm d}x
-\int_{g_0}^{h_0} v_0(x) {\rm d}x
=-\frac{dC_*}{\mu C_1}[h(t)-\psi(t)-2h_0]+\int_{0}^{t}\int_{g(s)}^{h(s)}  f(s,x,v) {\rm d} x{\rm d} s.
\end{align}

On the other hand, from \eqref{f2} we have
\begin{align*}
\frac{\partial}{\partial t} \lf[\int_{g(t)}^{h(t)} u(t,x) \rd x\rr]=&h'(t)v(t,h(t))-g'(t)v(t,g(t))+\int_{g(t)}^{h(t)} \partial_t v(t,x) \rd x\\
=&\int_{g(t)}^{h(t)} \partial_t v(t,x) \rd x
\end{align*}
and
\begin{align*}
&\int_{g(t)}^{h(t)} v_{xx}\rd x=v_x(t,h(t))-v_x(t,g (t))=-\frac 1 \mu[h'(t)-g'(t)].
\end{align*}
So similarly we have 
\begin{align*}
\int_{g(t)}^{h(t)} v(t,x){\rm d}x
-\int_{g_0}^{h_0} v_0(x) {\rm d}x
=-\frac{d}{\mu}[h(t)-g(t)-2h_0]+\int_{0}^{t}\int_{g(s)}^{h(s)}  f(s,x,v) {\rm d} x{\rm d} s.
\end{align*}
Comparing this identity with \eqref{limit-u}, we immediately obtain $C_1=C_*$.
\end{proof}

Next we examine the asymptotic limit of the solution $(u_\epsilon, h_\epsilon, g_\epsilon)$ of \eqref{f3'} with $C_1=C_*$, as $\epsilon\to 0$.

\begin{proposition}\label{p5.2}
Suppose the conditions of Theorem \ref{theorem1.1} hold,  and  $(u_\epsilon, g_\epsilon, h_\epsilon)$ is the solution of \eqref{f3'}
with $C_1=C_*$. Then there exists $\underline\mu>0$ such that
\[
\liminf_{\epsilon\to 0} u_\epsilon(t,x)\geq \underline v(t,x),\; \liminf_{\epsilon\to 0} h_\epsilon(t)\geq \underline h(t),\; \limsup_{\epsilon\to 0} g_\epsilon(t)\leq \underline
g (t)
\]
uniformly for $x\in \R$, $t\in [0, T]$ with every $T>0$, where $(\underline v(t,x), \underline g (t),\underline h(t))$ denotes the unique solution of \eqref{f2} with $\mu=\underline \mu$.
\end{proposition}

Here,  we assume that $u_\epsilon$ and $\ul v$ are extended by 0 outside their supporting sets.

 \begin{proof} For $0<\epsilon\ll 1$, let $(\underline {v_\epsilon}, \underline {g _\epsilon}, \underline {h_\epsilon})$ be the unique solution of \eqref{f4} with $\mu=\underline\mu$ and $i=2$. The value of $\underline \mu>0$ will be determined later.
 
 For an arbitrarily given $T_0>0$, fix $T=T_0$. By Lemma 3.2, the function
 \[
 \hat v_\epsilon(t,x):=\underline{v_\epsilon}(t,x)-\epsilon^{\gamma_1}Ke^{L_1 t}\ul{m_\epsilon}(t,x)
 \]
 satisfies (3.3), where $\ul{m_\epsilon}$ is given by (3.1) with $(\ul\mu, \underline {g _\epsilon}, \underline {h_\epsilon})$ in place of $(\mu, g_{2\epsilon}, h_{2\epsilon})$, and the same change is understood in (3.3). We assume that $\ul{v_\epsilon}$ has been extended to a $C^{0, 2+\alpha}([0, T_0]\times\R)$ function.
 
 Let $M(x)$ and $N(x)$ be smooth nonnegative functions over $[0,1]$
vanishing at $x\in \{0,1\}$. For convenience of notation, we define $M(x)=N(x)=0$ for $x\not\in [0,1]$.
Then define, for $0<\epsilon\ll 1$, $t\in (0, T_0]$ and $x\in [\underline {g_\epsilon}(t), \underline {h_\epsilon} (t)]$,
\[
v_\epsilon(t,x):=\hat v_\epsilon (t,x)+\epsilon M\Big(\frac{\underline {h_\epsilon} (t)-x}{\epsilon}\Big)\frac {\underline {h_\epsilon} '(t)}{\underline \mu}-\epsilon N\Big(\frac{x-\underline {g_\epsilon} (t)}{\epsilon}\Big)\frac{\underline {g_\epsilon}'(t)}{\underline \mu};
\]
and so we have, by (3.3),
\begin{align*}
&\mu\frac{C_*}{\epsilon^{2}} \int_{\ul{g_\epsilon}(t)}^{\ul{h_\epsilon}(t) }\int_{\ul{h_\epsilon}(t)}^{\yy}J_{{\epsilon}}(x-y)v_\epsilon(t,x)\rd y\rd x\\
=&\mu\frac{C_*}{\epsilon^{2}}\int^0_{-\epsilon}\int^{\epsilon}_0 J_\epsilon(x-y)v_\epsilon (t, \ul{h_\epsilon}(t)+x)dydx\\
=&\mu \frac{C_*}\epsilon \int^0_{-1}\int_{0}^1 J(x-y)v_\epsilon(t, \ul{h_\epsilon}(t)+\epsilon x)dydx\\
=&\mu \frac{C_*}\epsilon \int_{0}^1\int_{w}^1J(z)v_\epsilon(t, \ul{h_\epsilon}(t)-\epsilon w)dzdw\\
=&\mu \frac{C_*}\epsilon \int_{0}^1\int_{w}^1J(z)\big[-(\hat v_\epsilon)_x(t, \ul{h_\epsilon}(t))\epsilon w+ O(\epsilon^2w^2)+ \frac{\ul{h_\epsilon}'(t)}{\ul\mu}\epsilon M(w)\big]dzdw\\
\geq &\frac{\mu}{\ul{\mu}}\ul{h_\epsilon}'(t) C_* \int_{0}^1\int_{w}^1J(z)[w+M(w)]dzdw\\
=& \frac{\mu}{\ul{\mu}}\ul{h_\epsilon}'(t) C_* \int_{0}^1J(z)\int_{0}^z[w+M(w)]dwdz\\
=&  \frac{\mu}{\ul{\mu}}\ul{h_\epsilon}'(t)   \left[\frac 12 + C_*\int_{0}^1J(z)\int_{0}^z M(w)dwdz\right]\\
\geq & \ul{h_\epsilon}'(t)
\end{align*}
provided that
\[
\underline \mu\leq \mu\left[\frac 12 +C_*\int_{0}^1J(z)\int_{0}^z M(w)dwdz\right].
\]

Analogously, 
\begin{align*}
&\mu\frac{C_*}{\epsilon^{2}} \int_{\ul{g_\epsilon}(t)}^{\ul{h_\epsilon}(t) }\int_{-\infty}^{\ul{g_\epsilon}(t)}J_{{\epsilon}}(x-y)v_\epsilon(t,x)\rd y\rd x\\
\geq &\frac{\mu}{\ul{\mu}}|\ul{g_\epsilon}'(t)| C_* \int_{0}^1\int_{w}^1J(z)[w+N(w)]dzdw\\
=&  \frac{\mu}{\ul{\mu}}|\ul{g_\epsilon}'(t)|   \left[\frac 12 + C_*\int_{0}^1J(z)\int_{0}^z N(w)dwdz\right]\\
\geq & -\ul{g_\epsilon}'(t)
\end{align*}
provided that
\[
\underline \mu\leq \mu\left[\frac 12 +C_*\int_{0}^1J(z)\int_{0}^z N(w)dwdz\right].
\]

 For $0<\epsilon\ll1$, $t\in (0, T_0]$ and $x\in [\ul{g_\epsilon}(t), \ul{h_\epsilon}(t)]$, we further have
\bes\label{v-ep}
\begin{aligned}
&d\frac{C_*}{\epsilon^2}\!\! \lf[\int_{\ul{g_\epsilon}(t)}^{\ul{h_\epsilon}(t)}\hspace{-.2cm}J_\epsilon(x-y)v_\epsilon (t,y)\rd y-v_\epsilon(t,x)\rr] \\
&=d\frac{C_*}{\epsilon^2}\!\! \lf[\int_{\ul{g_\epsilon}(t)-\epsilon}^{\ul{h_\epsilon}(t)+\epsilon}\hspace{-.2cm}J_\epsilon(x-y)\hat v_\epsilon(t,y)\rd y-\hat v_\epsilon(t,x)\rr]-L_\epsilon(t,x)-R_\epsilon(t,x),
\end{aligned}
\ees
where
\begin{align*}
\dd L_\epsilon(t,x):=\ &d\frac{C_*}{\epsilon^2}\!\! \left\{\int_{\ul{g_\epsilon}(t)-\epsilon}^{\ul{g_\epsilon}(t)}\hspace{-.4cm}J_\epsilon(x-y)\hat v_\epsilon(t,y)\rd y\right.\\
& \ \ \ \ \ \ \ \ \ \left.+\epsilon \frac{\ul{g_\epsilon}'(t)}{\ul\mu}\Big[\!\int_{\ul{g_\epsilon}(t)}^{\ul{g_\epsilon}(t)+\epsilon}\hspace{-.5cm}J_\epsilon(x-y)N\Big(\frac{y-\ul{g_\epsilon}(t)}{\epsilon}\Big)dy-N\Big(\frac{x-\ul{g_\epsilon}(t)}{\epsilon}\Big)\!\Big]\!\right\},
\end{align*}
\begin{align*}
\dd R_\epsilon(t,x):=\ & d\frac{C_*}{\epsilon^2}\!\! \left\{\int^{\ul{h_\epsilon}(t)+\epsilon}_{\ul{h_\epsilon}(t)}\hspace{-.5cm}J_\epsilon(x-y)\hat v_\epsilon(t,y)\rd y\right.\\
& \ \ \ \ \ \ \ \ \ \left. -\epsilon \frac{\ul{h_\epsilon}'(t)}{\ul\mu}\Big[\!\int^{\ul{h_\epsilon}(t)}_{\ul{h_\epsilon}(t)-\epsilon}\hspace{-.5cm}J_\epsilon(x-y)M\Big(\frac{\ul{h_\epsilon}(t)-y}{\epsilon}\Big)dy-M\Big(\frac{\ul{h_\epsilon}(t)-x}{\epsilon}\Big)\!\Big]\!\right\}.
\end{align*}
From ${\rm spt}(J_\epsilon)\subset [-\epsilon, \epsilon]$ and ${\rm spt}(M),\; {\rm spt}(N)\subset [0,1]$, we see immediately that
\[
L_\epsilon(t,x)=0 \mbox{ for } x\not\in [\ul{g_\epsilon}(t), \ul{g_\epsilon}(t)+2\epsilon],\; R_\epsilon(t,x)= 0 \mbox{ for } x\not\in [\ul{h_\epsilon}(t)-2\epsilon, \ul{h_\epsilon}(t)],
\]
and due to $(\hat v_\epsilon)_x(t, \ul{h_\epsilon}(t))<0<(\hat v_\epsilon)_x(t, \ul{g_\epsilon} (t))$, that
\[
L_\epsilon(t,x)\leq 0 \mbox{ for } x\in [\ul{g_\epsilon} (t)+\epsilon, \ul{g_\epsilon}(t)+2\epsilon],\; R_\epsilon(t,x)\leq 0 \mbox{ for } x\in [\ul{h_\epsilon}(t)-2\epsilon, 
\ul{h_\epsilon}(t)-\epsilon].
\]
For $x\in  [\ul{h_\epsilon}(t)-\epsilon, \ul{h_\epsilon}(t)]$, letting
\[ z=\frac{y-\ul{h_\epsilon}(t)}{\epsilon} \mbox{ and } w=\frac{\ul{h_\epsilon}(t)-x}{\epsilon}\in [0,1],
\]
 we obtain
\begin{align*}
R_\epsilon(t,x)\!\!
=&d\frac{C_*}{\epsilon^2}\!\!  \left\{\int_0^{1}J(w+z)\hat v_\epsilon(t,\ul{h_\epsilon}+\epsilon z)d z-\epsilon\frac{\ul{h_\epsilon}'(t)}{\ul\mu}\left[\int_{-1}^0 J(w+z)M(-z)dz-M(w)\right]\right\}
\\
=&d\frac{C_*}{\epsilon^2}\!\!  \left\{\int_0^{1}\!\!\!J(w+z)\left[(\hat v_\epsilon)_x(t,\ul{h_\epsilon})\epsilon z +O(\epsilon^2z^2)\right] dz-\epsilon\frac{\ul{h_\epsilon}'(t)}{\ul\mu}\left[\int_{-1}^0\!\!\! J(w+z)M(-z)dz-M(w)\right]\right\}
\\
\leq &d\frac{C_*\ul{h_\epsilon}'(t)}{\epsilon\ul\mu}\!\!  \left\{-\int_0^{1}J(w+z)zdz -\int_0^1 J(w-z)M(z)dz+M(w)\right\}\\
&\;\; -\frac{dC_*}{\ul\mu}\int_0^1 J(w+z)zdz \left[\epsilon^{\gamma_1-1}+O(1)\right]
\\
\leq &d\frac{C_*\ul{h_\epsilon}'(t)}{\epsilon\ul\mu}\!\!  \left\{-\int_0^1\Big[J(w+z)z+J(w-z)M(z)\Big]dz +M(w)\right\}\\
&\;\; -\epsilon^{\gamma_1-1}\frac{dC_*}{2\ul\mu}\int_0^1 J(w+z)zdz.
\end{align*}

We now choose $M$ such that 

 \begin{equation}\label{M-condition}
\dd \int_0^1\Big[J(w-z)M(z)+J(w+z)z\Big]dz \geq  M(w) \mbox{ for } w\in [0, 1].
 \end{equation}
 This can be easily achieved. Indeed,
 for $w\in [0,1]$, write 
 \[
 F(w):= \int_0^1J(w+z)zdz.
 \]
 Then
 \[
 F(w)=\int_w^{1+w}J(\xi)(\xi-w)d\xi=\int_w^1 J(\xi)(\xi-w)d\xi,\]
 \[
  F(1)=0,\; F(0)=\int_0^1 J(\xi)\xi d\xi\in (0, \frac 12),
 \]
 and
 \[
 F'(w)=-\int_w^1J(\xi)d\xi\leq 0,\; F'(0)=-\frac 12, \; F'(1)=0,\; F''(w)=J(w)\geq 0 \mbox{ for } w\in [0,1].
 \]
 It follows that
 \[
 F(0)-\frac 12 w\leq F(w)<\frac 12 (1-w) \mbox{ for } w\in [0,1).
 \]
 Therefore we can find $M$ satisfying, apart from the earlier requirements,  that
 \[\mbox{
  $M(w)\leq F(w)$ in $[0,1]$, $M(w)=F(w)$ for $w\in [F(0), 1]$.}
  \]
 
 Fix such an $M$; then clearly \eqref{M-condition} holds and we have
 \[
 R_\epsilon(t,x)\leq -\epsilon^{\gamma_1-1}\frac{dC_*}{2\ul\mu}F(w)
 \]
 for $0<\epsilon\ll 1$ and $w=\frac{\ul{h_\epsilon}(t)-x}{\epsilon}\in [0,1]$.
 
  Similarly,
  \begin{align*}
L_\epsilon(t,x)
\leq &\; d\frac{C_*|\ul{g_\epsilon}'(t)|}{\epsilon\mu}\!\!  \left\{-\int_0^1\Big[J(w+z)z+J(w-z)N(z)\Big]dz +N(w)\right\}\\
&\; -\epsilon^{\gamma_1-1}\frac{dC_*}{2\ul\mu}\int_0^1 J(w+z)zdz.
\end{align*}
So if we take $N(w)=M(w)$, then
 \[
 L_\epsilon(t,x)\leq -\epsilon^{\gamma_1-1}\frac{dC_*}{2\ul\mu}F(w)
 \]
 for $0<\epsilon\ll 1$ and $w=\frac{x-\ul{g_\epsilon}(t)}{\epsilon}\in [0,1]$.

With $M=N$ defined as above, we now take
\[
\ul\mu:=\mu\left[\frac 12 +C_*\int_{0}^1J(z)\int_{0}^z M(w)dwdz\right].
\]
Then
\[
\begin{cases}
\dd\mu\frac{C_*}{\epsilon^{2}} \int_{\ul{g_\epsilon}(t)}^{\ul{h_\epsilon}(t) }\int_{\ul{h_\epsilon}(t)}^{\yy}J_{{\epsilon}}(x-y)v_\epsilon(t,x)\rd y\rd x\geq  \ul{h_\epsilon}'(t),\\
\dd \mu\frac{C_*}{\epsilon^{2}} \int_{\ul{g_\epsilon}(t)}^{\ul{h_\epsilon}(t) }\int_{-\infty}^{\ul{g_\epsilon}(t)}J_{{\epsilon}}(x-y)v_\epsilon(t,x)\rd y\rd x\geq  -\ul{g_\epsilon}'(t),
\end{cases} 
\]
for $0<\epsilon\ll 1$ and $t\in [0, T_0]$.

By Steps 4 and 5 in the proof of Lemma 3.3, there exists $c_0>0$ such that
\begin{equation}\label{hat-v}
\partial_t \hat v_\epsilon\leq d\frac{C_*}{\epsilon^2}\!\! \lf[\int_{\ul{g_\epsilon}(t)-\epsilon}^{\ul{h_\epsilon}(t)+\epsilon}\hspace{-.2cm}J_\epsilon(x-y)\hat v_\epsilon(t,y)\rd y-\hat v_\epsilon(t,x)\rr] +f(t,x,\hat v_\epsilon)-c_0\epsilon^{\gamma_1}
\end{equation}
for $0<\epsilon\ll 1$, $t\in [0, T_0]$ and $x\in [\ul{g_\epsilon}(t), \ul{h_\epsilon}(t)]$.

We show next that
\bes\label{v-ep-req}
\partial_t  v_\epsilon\leq d\frac{C_*}{\epsilon^2}\!\! \lf[\int_{\ul{g_\epsilon}(t)}^{\ul{h_\epsilon}(t)}\hspace{-.2cm}J_\epsilon(x-y) v_\epsilon(t,y)\rd y-v_\epsilon(t,x)\rr] +f(t,x,v_\epsilon)
\ees
for $0<\epsilon\ll 1$, $t\in [0, T_0]$ and $x\in [\ul{g_\epsilon}(t), \ul{h_\epsilon}(t)]$.

When $x\in [\ul{g_\epsilon}(t)+\epsilon, \ul{h_\epsilon}(t)-\epsilon]$, we have $v_\epsilon(t,x)=\hat v_\epsilon(t,x)$, and by \eqref{v-ep}, \eqref{hat-v}, we obtain
\begin{align*}
&d\frac{C_*}{\epsilon^2}\!\! \lf[\int_{\ul{g_\epsilon}(t)}^{\ul{h_\epsilon}(t)}\hspace{-.2cm}J_\epsilon(x-y)v_\epsilon (t,y)\rd y-v_\epsilon(t,x)\rr] \\
&\geq d\frac{C_*}{\epsilon^2}\!\! \lf[\int_{\ul{g_\epsilon}(t)-\epsilon}^{\ul{h_\epsilon}(t)+\epsilon}\hspace{-.2cm}J_\epsilon(x-y)\hat v_\epsilon(t,y)\rd y-\hat v_\epsilon(t,x)\rr]
\\
&\geq \partial v_\epsilon-f(t,x,v_\epsilon),
\end{align*}
as we wanted.

For $x\in [\ul{g_\epsilon}(t), \ul{g_\epsilon}(t)+\epsilon]$,
by \eqref{v-ep} and \eqref{hat-v}, we obtain
\bes\label{5.6}
\begin{aligned}
&d\frac{C_*}{\epsilon^2}\!\! \lf[\int_{\ul{g_\epsilon}(t)}^{\ul{h_\epsilon}(t)}\hspace{-.2cm}J_\epsilon(x-y)v_\epsilon (t,y)\rd y-v_\epsilon(t,x)\rr] \\
&\geq \partial \hat v_\epsilon-f(t,x,\hat v_\epsilon)+c_0\epsilon^{\gamma_1}+\epsilon^{\gamma_1-1}\frac{dC_*}{2\ul\mu}F(\frac{x-\ul{g_\epsilon}(t)}{\epsilon}),
\end{aligned}
\ees
and by the Lipschitz continuity of $f$ we also have
\[
-f(t,x,\hat v_\epsilon)\geq -f(t,x,v_\epsilon)-O(\epsilon).
\]
Moreover, due to our choice of $M$, 
\begin{align*}
\partial_t \hat v_\epsilon=&\ \partial_t v_\epsilon+\epsilon M(\frac{x-\ul{g_\epsilon}(t)}{\epsilon})\frac{\ul{g_\epsilon}''(t)}{\ul\mu}-M'(\frac{x-\ul{g_\epsilon}(t)}{\epsilon})\frac{[\ul{g_\epsilon}(t)]^2}{\ul\mu}\\
&\ \geq \begin{cases} \partial_t v_\epsilon-O(\epsilon) & \mbox{ for } x\in [\ul{g_\epsilon}(t)+\epsilon F(0),\  \ul{g_\epsilon}(t)+\epsilon],\\
\partial_t v_\epsilon -O(1) & \mbox{ for } x\in [\ul{g_\epsilon}(t), \ \ul{g_\epsilon}(t)+\epsilon F(0)],
\end{cases}
\end{align*}
and
\[
\epsilon^{\gamma_1-1}\frac{dC_*}{2\ul\mu}F(\frac{x-\ul{g_\epsilon}(t)}{\epsilon})\geq c_1\epsilon^{\gamma_1-1} \mbox{ for } x\in [\ul{g_\epsilon}(t), \ \ul{g_\epsilon}(t)+\epsilon F(0)],
\]
where
\[
c_1:=\frac{dC_*}{2\ul\mu} \min_{w\in [0, F(0)]} F(w)>0.
\]
Substituting these estimates to \eqref{5.6}, we obtain
\[
d\frac{C_*}{\epsilon^2}\!\! \lf[\int_{\ul{g_\epsilon}(t)}^{\ul{h_\epsilon}(t)}\hspace{-.2cm}J_\epsilon(x-y)v_\epsilon (t,y)\rd y-v_\epsilon(t,x)\rr]
\geq \partial_t v_\epsilon-f(t,x,v_\epsilon)
\]
for $0<\epsilon\ll 1$, $t\in [0, T_0]$ and $x\in [\ul{g_\epsilon}(t), \ul{g_\epsilon}(t)+\epsilon]$.
Thus \eqref{v-ep-req} holds in this range of the variables. For $x\in [\ul{h_\epsilon}(t)-\epsilon, \ul{h_\epsilon}(t)]$, the proof is parallel and we omit the details.
Therefore \eqref{v-ep-req} holds for $0<\epsilon\ll 1$, $t\in [0, T_0]$ and $x\in [\ul{g_\epsilon}(t), \ul{h_\epsilon}(t)]$, as desired.

Since $0<v_\epsilon(0,x)\leq v_0(x)$ for $0<\epsilon\ll 1$ and $x\in (-h_0, h_0)$, we may now use the comparison principle to conclude that
\bes\label{5.8}
u_\epsilon(t,x)\geq v_\epsilon(t,x),\; h_\epsilon(t)\geq \ul{h_\epsilon}(t),\; g_\epsilon(t)\leq \ul{g_\epsilon}(t)  \mbox{ for } t\in [0, T_0],\; x\in [\ul{g_\epsilon}(t), \ul{h_\epsilon}(t)].
\ees
By Theorem 2.4, we have $(\ul{v_\epsilon},\ul{g_\epsilon},\ul{h_\epsilon})\to (\ul v, \ul g, \ul h)$ uniformly as $\epsilon\to 0$, and hence
$v_\epsilon\to \ul v$ uniformly as $\epsilon\to 0$. The required estimates now follow directly by letting $\epsilon\to 0$ in \eqref{5.8}.
 \end{proof}
 \begin{remark}\begin{itemize}{\rm
 \item[(i)]Note that if for some kernel function $J$ we can find a function $M$ satisfying \eqref{M-condition} and
 \[
 \int_{0}^1J(z)\int_{0}^z M(w)dwdz>\frac 12 C_*^{-1}=\frac 12 \int_0^1J(z) z^2dz,
 \]
 then from the above proof we see that $\ul\mu>\mu$, and therefore the conclusion in Proposition 5.2 implies that $(u_\epsilon, g_\epsilon, h_\epsilon)$ cannot  converge to $(v,g,h)$, the unique solution of \eqref{f2}, uniformly as $\epsilon\to 0$, since $\ul \mu>\mu$ implies $\ul{v}>v$. However, we have not been able to find such a pair $(J, M)$ so far, though we suspect such a pair exists.
 \item[(ii)] It is also possible to obtain an estimate for $(u_\epsilon, g_\epsilon, h_\epsilon)$ in the form
 \[
\limsup_{\epsilon\to 0} u_\epsilon(t,x)\leq \overline v(t,x),\; \limsup_{\epsilon\to 0} h_\epsilon(t)\leq \overline h(t),\; \liminf_{\epsilon\to 0} g_\epsilon(t)\geq \overline
g (t)
\]
uniformly for $x\in \R$, $t\in [0, T]$ with every $T>0$, where $(\overline v(t,x), \overline g (t),\overline h(t))$ denotes the unique solution of \eqref{f2} with $\mu=\overline \mu$.}
\end{itemize}
 \end{remark}

\end{document}